\documentclass[11pt]{article}
\usepackage{amssymb,amsmath,bm}
\usepackage{amsthm}
\usepackage{xcolor}
\usepackage{textcomp}
\usepackage{enumerate}      
\usepackage{graphicx}        
\usepackage{caption, subcaption}

\usepackage{url}

\usepackage{tabu}

\usepackage[export]{adjustbox}

\usepackage{mathrsfs}

\usepackage{url}

\usepackage{array}
\newcommand{\PreserveBackslash}[1]{\let\temp=\\#1\let\\=\temp}
\newcolumntype{C}[1]{>{\PreserveBackslash\centering}p{#1}}
\newcolumntype{R}[1]{>{\PreserveBackslash\raggedleft}p{#1}}
\newcolumntype{L}[1]{>{\PreserveBackslash\raggedright}p{#1}}

\renewcommand{\setminus}{{\smallsetminus}}

\usepackage{amssymb,amsmath,bm}
\usepackage{amsthm}
\usepackage{hyperref}
\usepackage{mathrsfs}
\usepackage{textcomp}
\usepackage{enumerate}
\usepackage{graphicx}
\usepackage{tikz, tikz-cd}
\usepackage{verbatim}

\newtheorem{theorem}{Theorem}[section]
\newtheorem{lemma}[theorem]{Lemma}
\newtheorem{proposition}[theorem]{Proposition}
\newtheorem{definition}[theorem]{Definition}
\newtheorem{corollary}[theorem]{Corollary}

\theoremstyle{remark}
\newtheorem{remark}[theorem]{Remark}
\newcommand{\R}{\mathbb{R}}

\theoremstyle{remark}

\numberwithin{equation}{section}

\begin{document}

\title{\bf $\mathrm {U}_{q\tilde q}\mathfrak{sl}(2;\mathbb R)$ Turaev-Viro invariants for cusped  $3$-manifolds}
\date{}

\author{Tianyue Liu} 
\author{Shuang Ming} 
\author{Xin Sun} 
\author{Baojun Wu}
\author{Tian Yang}
\author{Tianyue Liu, Shuang Ming, Xin Sun, Baojun Wu, and Tian Yang}

\maketitle

\begin{abstract} 
We define a family of Turaev-Viro type invariants for hyperbolic $3$-manifolds with cusp ends, extending the invariants introduced in \cite{LMSWY} for hyperbolic $3$-manifolds with totally geodesic boundary. These invariants are constructed from what we call the ideal $\mathrm{U}_{q\tilde q}\mathfrak{sl}(2;\mathbb R)$-$6j$ symbols, which are variants of the $6j$-symbols associated with the positive representations of the modular double of $\mathrm{U}_q\mathfrak{sl}(2;\mathbb R)$. We also prove that these invariants decay exponentially, with the exponential decay rate determined by the hyperbolic volume of the manifold.
\end{abstract}

\tableofcontents


\section{Introduction}

A recurring theme in quantum topology is that the asymptotic behavior of
quantum invariants encodes the geometry of $3$-manifolds. This principle
first appeared in Kashaev's volume conjecture, which relates the
exponential growth of a knot invariant defined from the quantum
dilogarithm to the hyperbolic volume of the knot complement
\cite{Kashaev}. Murakami and Murakami identified Kashaev's invariant
with a specialization of the colored Jones polynomial
\cite{MurakamiMurakami}. Analogous volume conjectures were subsequently
proposed by Chen and Yang for the Reshetikhin-Turaev and Turaev-Viro
invariants of hyperbolic $3$-manifolds \cite{ChenYang}, and were verified
in important examples in 
\cite{DetcherryKalfagianniYang, BDKY, O, WY}. These developments initiated a broad
program relating the asymptotics of quantum invariants to hyperbolic
volume and other geometric invariants of $3$-manifolds.

Recently, we initiated the study of a quantum invariant based on a natural non-compact quantum group $\mathrm {U}_{q\tilde q}\mathfrak{sl}(2;\mathbb R)$, which is the modular double of
$\mathrm {U}_{q}\mathfrak{sl}(2;\mathbb R)$.
Its $6j$-symbols are the Racah-Wigner coefficients of a continuous
series of the positive representations.  In our previous
work \cite{LMSWY}, we used these $6j$-symbols as tetrahedral weights
to construct Turaev-Viro type invariants of hyperbolic $3$-manifolds
with totally geodesic boundary.  We proved that their exponential decay
rate is the hyperbolic volume and identified the one-loop term with the
adjoint twisted Reidemeister torsion of the double of the manifold.

The purpose of the present paper is to extend our $\mathrm U_{q\tilde q}\mathfrak{sl}(2;\mathbb R)$ Turaev-Viro theory to hyperbolic $3$-manifolds with cusp
ends.  This is a key case for the program above, since cusped
manifolds -- and knot complements in particular -- form the original
geometric setting of the volume conjecture.  Such an
extension is nontrivial because the original state integral diverges in the cusp case. 
To overcome this issue, we introduce an ideal
$\mathrm U_{q\tilde q}\mathfrak{sl}(2;\mathbb R)$-$6j$ symbol as a
renormalized ideal limit of the ordinary $\mathrm {U}_{q\tilde q}\mathfrak{sl}(2;\mathbb R)$-$6j$ symbol. Using
these new tetrahedral weights, we construct a state integral for cusped
hyperbolic $3$-manifolds and prove that it decays exponentially with the decay rate
given by the hyperbolic volume.

Our construction is closely related to several non-compact state-integral
models based on Faddeev's quantum dilogarithm, quantum Teichm\"uller
theory and conformal field theory (CFT). On the mathematical side,  Andersen and Kashaev constructed the so-called Teichm\"uller TQFT, while
Kashaev, Luo, and Vartanov constructed a Turaev-Viro type partition
function on shaped triangulations \cite{AndersenKashaev,KashaevLuoVartanov}.  
On the physics side, the Virasoro TQFT proposed
by Collier, Eberhardt, and Zhang~\cite{CollierEberhardtZhang} provides a formulation of
three-dimensional quantum gravity with negative cosmological constant
in terms of Virasoro crossing kernels in CFT, which are closely related
to the $\mathrm U_{q\tilde q}\mathfrak{sl}(2;\mathbb R)$-$6j$ symbols. We believe that there exist close connections among all the non-compact theories. However, our construction itself is quite different from existing ones, and establishing a precise unifying picture is a long-term goal of our program. 

The rest of the introduction explains our construction and states our main results.

\subsection{Ideal $\mathrm {U}_{q\tilde q}\mathfrak{sl}(2;\mathbb R)$-$6j$ symbols}

For $b\in (0,1),$ let $Q=b+\frac{1}{b}.$ For $z\in \mathbb C,$ define the function 
$$T_b(z)\doteq e^{-\frac{\pi\mathbf i}{2}\big(z^2-Qz+\frac{Q^2+1}{6}\big)}$$
with the convention that $T_b(z)^{\frac{1}{2}}\doteq e^{-\frac{\pi\mathbf i}{4}\big(z^2-Qz+\frac{Q^2+1}{6}\big)}.$  For $z\in \mathbb C$ with $0<\mathrm{Re}(z)<Q,$ let $S_b$ be the double sine function defined by
\begin{equation*}
S_b(z)=\exp\Bigg(\int_\Omega\frac{\sinh\Big(\big(\frac{Q}{2}-z\big)t\Big)}{4t\sinh(\frac{bt}{2})\sinh(\frac{t}{2b})}dt\Bigg),
\end{equation*}
where the contour $\Omega$ goes along the real line and passes above the pole of the integrand at the origin.

For a six-tuple $(a_1,\dots,a_6)\in \big( \frac{Q}{2}+ \mathbf i \mathbb R\big)^6,$ let  
\begin{equation*}
\begin{split}
   & t_1=a_1+a_2+a_3,\quad  t_2=a_1+a_5+a_6,\quad t_3=a_2+a_4+a_6,\quad  t_4=a_3+a_4+a_5,\\
   & q_1=a_1+a_2+a_4+a_5,\quad q_2=a_1+a_3+a_4+a_6 \quad\text{and}\quad   
q_3=a_2+a_3+a_5+a_6.
\end{split}
\end{equation*}

\begin{definition}\label{def:ideal 6j}

The \emph{ideal $\mathrm {U}_{q\tilde q}\mathfrak{sl}(2;\mathbb R)$-$6j$ symbol} with parameter $\boldsymbol a=(a_1,\dots,a_6)$ is given by
\begin{equation}\label{b-6j}
\bigg|\begin{matrix} a_1 & a_2 & a_3 \\ a_4 & a_5 & a_6 \end{matrix} \bigg|_b=\Bigg(\frac{\prod_{i=1}^4 T_b(2Q-t_i)}{\prod_{i=1}^4\prod _{j=1}^3 T_b(q_j-t_i)}\Bigg)^{\frac{1}{2}}\int_\Gamma \frac{\prod_{i=1}^4 T_b(u-t_i)}{T_b(2Q-u)}\prod_{j=1}^3S_b(q_j-u)d\mathrm{Im}u,
\end{equation}
where $\Gamma$ is any vertical line passing the interval $\big(\frac{11Q}{6},2Q\big),$
oriented upward.
\end{definition}

This definition comes from the ``idealization" of the original $\mathrm {U}_{q\tilde q}\mathfrak{sl}(2;\mathbb R)$-$6j$ symbols. See Lemma \ref{AI}. In Propositions \ref{lem:ab-conver} and \ref{32}, we will show that the integral in (\ref{b-6j}) converges absolutely and is independent of the choice of the vertical line $\Gamma,$ and the ideal $\mathrm {U}_{q\tilde q}\mathfrak{sl}(2;\mathbb R)$-$6j$ symbols satisfy a suitable Pentagon Identity. 
\bigskip

In the rest of this paper, for $(x_1,\dots,x_6)\in\mathbb R^6$ and for each  $k\in\{1,\dots,6\},$ we will let 
$$a_k = \frac{Q}{2} + \mathbf i \frac{x_k}{2\pi b}.$$

\begin{theorem}\label{6jasymp} 
If $(x_1,\dots, x_6)\in\mathbb R^6$ are the edge lengths of a decorated ideal hyperbolic tetrahedron $\Delta,$  then as $b\to 0,$
$$\bigg|\begin{matrix} a_1 & a_2 & a_3 \\ a_4 & a_5 & a_6 \end{matrix} \bigg|_b=\frac{1}{2}\frac{e^{\frac{-\mathrm{Cov}(\Delta)}{\pi b^2}}}{\sqrt[4]{-\det\mathrm{Gram}(\Delta)}}\Big(1+O\big(b^2\big)\Big),$$
where $\mathrm{Cov}(\Delta)$ is the co-volume of $\Delta$, and  $\mathrm{Gram}(\Delta)$ is the Gram matrix of $\Delta$ in the edge lengths.
\end{theorem}

\begin{remark}
Using the same argument as in \cite{LMSWY, LMSWY2}, one can obtain the following results, whose proof is left to the interested readers.
\begin{enumerate}[(1)]
\item Every $(x_1,\dots, x_6)\in \mathbb R^6$ is exclusively the six-tuple of the edge lengths of:
\begin{enumerate}[(i)]
\item a decorated ideal hyperbolic tetrahedron in $\mathbb H^3,$ this is when the Gram matrix has signature $(3,1),$ 
\item a (decorated ideal) flat tetrahedron, this is when the Gram matrix has signature $(2,1),$ 
\item a decorated ideal hyperbolic tetrahedron in  $\mathbb A\mathrm d\mathbb S^3,$ this is when the Gram matrix has signature $(2,2).$   
\end{enumerate}

\item If $(x_1,\dots, x_6)\in{\mathbb R^6}$ are the edge lengths of a flat tetrahedron $\Delta,$ then as $b\to 0,$
$$\lim_{b\to 0}\pi b^2\log\bigg|\begin{matrix} a_1 & a_2 & a_3 \\ a_4 & a_5 & a_6 \end{matrix}  \bigg|_b=-\widetilde{\mathrm{Cov}}(\Delta),$$
where $\widetilde{\mathrm{Cov}}(\Delta)$ is the extended co-volume function of $(x_1,\dots,x_6)$ defined in \cite{LY}.

\item If $(x_1,\dots, x_6)\in\mathbb R^6$ are the edge lengths of a decorated ideal anti-de Sitter tetrahedron $\Delta,$ then  as $b\to 0,$
\begin{equation*}\label{AdS/CFT}
\bigg|\begin{matrix} a_1 & a_2 & a_3 \\ a_4 & a_5 & a_6 \end{matrix}  \bigg|_b=\frac{e^{\frac{-\widetilde{\mathrm{Cov}}(\Delta)}{\pi b^2}} }{\sqrt[4]{\det\mathrm{Gram}(\Delta)}} \Bigg(\cos\bigg( {\frac{\mathrm{Cov}_{\mathbb A\mathrm d\mathbb S^3}(\Delta)}{\pi b^2}}-\frac{\pi}{4}\bigg) +O\big(b^2\big)\Bigg),
\end{equation*}
where $\widetilde{\mathrm{Cov}}(\Delta)$ is the extended co-volume function defined in \cite{LY}, $\mathrm{Cov}_{\mathbb A\mathrm d\mathbb S^3}(\Delta)$ is the anti-de Sitter co-volume of $\Delta$, and $\mathrm{Gram}(\Delta)$ is the Gram matrix of $\Delta$ in the edge lengths.
\end{enumerate}
\end{remark}

\subsection{$\mathrm {U}_{q\tilde q}\mathfrak{sl}(2;\mathbb R)$ Turaev-Viro invariants for  hyperbolic  $3$-manifolds with cusp ends}
We now define the $\mathrm {U}_{q\tilde q}\mathfrak{sl}(2;\mathbb R)$ Turaev-Viro type invariant for hyperbolic $3$-manifolds with cusp ends.  As mentioned above, the construction of \cite{LMSWY}, based on the $\mathrm{U}_{q\tilde q}\mathfrak{sl}(2;\mathbb R)$-$6j$ symbols, does not extend directly to the cusped setting, as the corresponding state integral diverges. To overcome this difficulty, we instead use the ideal $\mathrm{U}{q\tilde q}\mathfrak{sl}(2;\mathbb R)$-$6j$ symbols introduced above.

Let  $M$ be a  hyperbolic $3$-manifold with cusp ends, and let $\mathcal T$ be an ideal triangulation of $M$ with the set of edges $E$ and the set of tetrahedra $T.$ A $b$-coloring of $(M,\mathcal T)$ is an assignment of a complex number $a_e$ of the form $\frac{Q}{2}+\mathbf i\frac{x_e}{2\pi b}$ with $x_e\in\mathbb R$ to each  $e$ of $\mathcal T.$ Let $\boldsymbol a=\big(a_e\big)_{e\in E}.$ For each $e\in E,$ define
\begin{equation}
  |e|_{\boldsymbol a}=|T_b(2a_e)|^2=e^{\frac{Qx_e}{b}};
\end{equation}
and for each $\Delta\in T,$ define
$$|\Delta|_{\boldsymbol a}=\bigg|\begin{matrix} a_{e_1} & a_{e_2} & a_{e_3} \\ a_{e_4} & a_{e_5} & a_{e_6} \end{matrix} \bigg|_b=\bigg|\begin{matrix} \frac{Q}{2}+\mathbf i\frac{x_{e_1}}{2\pi b} & \frac{Q}{2}+\mathbf i\frac{x_{e_2}}{2\pi b} & \frac{Q}{2}+\mathbf i\frac{x_{e_3}}{2\pi b} \\ \frac{Q}{2}+\mathbf i\frac{x_{e_4}}{2\pi b} & \frac{Q}{2}+\mathbf i\frac{x_{e_5}}{2\pi b} & \frac{Q}{2}+\mathbf i\frac{x_{e_6}}{2\pi b} \end{matrix} \bigg|_b,$$
where $\{e_1,\dots,e_6\}$ are the edges of $\mathcal T$ adjacent to $\Delta$ so that $e_1,e_2,e_3$ are the edges of a face of $\Delta$ and, for $i\in\{1,2,3\},$  $e_i$ and $e_{i+3}$ are opposite to each other.

Let $V$ be the set of boundary components of $M,$ and let the $|E|\times |V|$-matrix $A$ (also considered as a linear transformation from $\mathbb R^V$ to $\mathbb R^E$) be the adjacency matrix of the edges and the boundary components given by $A_{e,v}=|e\cap v|,$ the number of points of intersections of the edge $e$ with the boundary component $v.$ The matrix $A$ defines an action of $\mathbb R^V$ on $\mathbb R^E$ by 
$\boldsymbol u\cdot\boldsymbol x=\boldsymbol x + A\boldsymbol u,$ geometrically corresponding to the action of the change of horospheres. This in turn defines an $\mathbb R^V$-action on $\big(\frac{Q}{2}+\mathbf i \mathbb R\big)^E$ by $\boldsymbol u\cdot \boldsymbol a = \boldsymbol a + \frac{\mathbf i}{2\pi b} A \boldsymbol u;$ and we denote the quotient space by  $\big(\frac{Q}{2}+\mathbf i \mathbb R\big)^E/\mathbb R^V.$ Let $d
\mu$ be the quotient measure obtained by disintegrating the Lebesgue measure $\prod_{e\in E} d\mathrm{Im}a_e$ on $\big(\frac{Q}{2}+\mathbf i\mathbb R\big)^ E$ along the orbits $[\boldsymbol a]= \boldsymbol a + \frac{\mathbf i}{2\pi b} A(\mathbb R^V)$ of the action.

We consider the following integral 
\begin{equation}\label{b-TV}
    \mathrm{TV}_b(M,\mathcal T)\doteq \int_{\big(\frac{Q}{2}+\mathbf i\mathbb R\big)^ E/\mathbb R^V} \prod_{e\in E} |e|_{\boldsymbol{a}}\prod_{\Delta\in T}|\Delta|_{\boldsymbol a}d\mu([\boldsymbol a]).
\end{equation}

Theorem \ref{converge} below considers the convergence of the integral (\ref{b-TV}), under a natural condition that the ideal triangulation $\mathcal T$ supports angle structures (see Section \ref{AS} for the definition).  It was proved in \cite{HRS} that hyperbolic $3$-manifolds satisfying a mild topological condition - including all link complements in $\mathbb S^3$ - admit ideal triangulations supporting angle structures. Moreover, \cite{HRST} proved that Agol's veering triangulations admit angle structures. It has also been verified that every hyperbolic $3$-manifold with cusp ends in the SnapPy census admits an ideal triangulation supporting an angle structure. Therefore, Theorem \ref{converge} applies to all of these examples. 

\begin{theorem}\label{converge}
Let $M$ be a hyperbolic $3$-manifold with cusp ends, and let $\mathcal T$ be an ideal triangulation of $M$ that supports angle structures. Then for each $b\in (0,1),$ the integral $\mathrm{TV}_b(M,\mathcal T)$ converges absolutely.
\end{theorem}

The Pentagon Identity satisfied by the ideal
$\mathrm{U}_{q\tilde q}\mathfrak{sl}(2;\mathbb R)$-$6j$ symbols (Proposition \ref{pentagon}) implies the
invariance of $\mathrm{TV}_b(M,\mathcal T)$ under a $2$-$3$ or $3$-$2$
Pachner move between ideal triangulations that support angle structures,
as shown in Theorem~\ref{topinv} below. We believe that, for any
ideal triangulation $\mathcal T$ supporting angle structures, the value
of $\mathrm{TV}_b(M,\mathcal T)$ is independent of $\mathcal T$, and
hence defines a topological invariant of $M$.

\begin{theorem}\label{topinv}
Let $M$ be a hyperbolic $3$-manifold with cusp ends, and let 
$\mathcal T_1$ and $\mathcal T_2$ be two ideal triangulations of $M$ supporting angle structures that differ by a $2$-$3$ or $3$-$2$ Pachner Move.  Then 
$$\mathrm{TV}_b(M,\mathcal T_1)=\mathrm{TV}_b(M,\mathcal T_2).$$
\end{theorem}

Theorem \ref{vol} below considers the asymptotics of $\mathrm{TV}_b(M,\mathcal T),$ under the assumption that the ideal triangulation $\mathcal T$ is geometric, that is, each tetrahedron in $\mathcal T$ is a positively oriented hyperbolic ideal tetrahedron. 
It was conjectured by Thurston that every hyperbolic $3$-manifold with cusp ends has a geometric triangulation. Although still open, this conjecture has been verified for all  examples in the SnapPy census. Consequently, Theorem \ref{vol} applies to all these $3$-manifolds.

\begin{theorem}\label{vol}
Let $M$ be a hyperbolic $3$-manifold with cusp ends, and let $\mathcal T$ be a geometric ideal triangulation of $M.$ Then as $b\to 0,$
$$\lim_{b\to 0}\pi b^2\log \mathrm{TV}_b(M,\mathcal T)=-\mathrm{Vol}(M),$$
where $\mathrm{Vol}(M)$ is the hyperbolic volume of $M.$
\end{theorem}

\subsection{Forthcoming works}

\noindent{\bf Asymptotic expansion and adjoint twisted Reidemeister torsion.}
\bigskip
In a forthcoming work, we will show that 
$$\mathrm{TV}_{b}(M, \mathcal T)=\frac{1}{(2\pi b)^{|V|}}\frac{e^{\frac{-\mathrm{Vol}(M)}{\pi b^2}}}{\sqrt{|\mathbb T(M)|}}\big(1+O(b^{2})\big),$$
and
$$\mathbb T(M)=\mathrm{Im}\left(\det\left(\frac{\partial H(l_{i})}{\partial H(m_{j})}\right)_{1\leqslant i, j\leqslant  |V|}\right)\big|\mathrm{Tor}(M;\boldsymbol m)\big|^{2},$$ 
where for each $i\in \{1,\dots, |V|\},$ $m_{i}$ and $l_{i}$ are two simple closed curves on the $i$-th boundary component of $M$ that algebraically intersect once, $H(m_{i})$ and $H(l_{i})$ are respectively the logarithmic holonomies of $m_{i}$ and $l_{i}$ considered as functions over the deformation space of $(M, \mathcal{T}),$ and $\mathrm{Tor}(M; \boldsymbol m)$ is up to sign the adjoint twisted Reidemeister torsion of $M$ with $\boldsymbol m=(m_1,\dots,m_{|V|})$ the chosen system of  boundary curves. It can be verified that $\mathbb T(M)$ is independent of choice of $m_{i}$ and $l_{i}.$

\bigskip

\noindent{\bf Relationship with Andersen-Kashaev invariants.} In a forthcoming work, we will  for every ideally triangulated hyperbolic 3-manifold $ (M,\mathcal T)$  with cusp ends  define a family of Jones functions $\mathrm{J}_{b}(M,\mathcal T; \boldsymbol x),$ indexed by $\boldsymbol x\in\mathbb R^V,$ where $V$ is the set of boundary components of $M;$  and we will prove that 
\begin{equation}\label{eq:main-H}
 \mathrm{TV}_b(M,\mathcal{T})
 =
 \int_{\mathbb R^V}
 \mathrm{J}_{b}(M,\mathcal T; \boldsymbol x)
 \overline{\mathrm{J}_{b}(M,\mathcal T; -\boldsymbol x)}
 d\boldsymbol x
\end{equation}
when $\mathcal T$ supports angle structures. 
We will also show that our Jones functions are  essentially the Jones functions  predicted to show up in the Andersen-Kashaev partition function\,\cite{AndersenKashaev, BW}. Relation (\ref{eq:main-H}) can be considered as an analogue of the relationship between the original Turaev-Viro invariants and the colored Jones polynomials~\cite{DetcherryKalfagianniYang, BDKY}.

\bigskip

\noindent\textbf{Acknowledgments.} We thank Ka Ho Wong for helpful discussions. T.L.\ is supported by National Key R\&D Program of China (No.\ 2023YFA1010700) and by Simons Foundation  (SFI-MPS-PP-00012621-16).
S.M.\ is supported by National Natural Science Foundation of China (No.12371124).
X.S. and B.W. are supported by National Key R\&D Program of China (No.\ 2023YFA1010700) and National Natural Science Foundation of China Grant (No. 12526204).
 T.Y.\ is supported by NSF Grant  DMS-2505908.


\section{Preliminaries}

\subsection{Decorated ideal hyperbolic tetrahedra} \label{idealtetra}

Let $\mathbb H^3$ be the $3$-dimensional hyperbolic space with  $\partial \mathbb H^3$  its boundary. An \emph{ideal hyperbolic tetrahedron} is the convex hull of four points $\{v_1,v_2,v_3,v_4\}$ on $\partial \mathbb H^3$ that do not lie in the same totally geodesic plane in $\mathbb H^3.$ The points $\{v_1,v_2,v_3,v_4\}$ are the \emph{ideal vertices} of the ideal hyperbolic tetrahedron. The \emph{edge} $e_{ij}$ of the ideal hyperbolic tetrahedron is the geodesic connecting $v_i$ and $v_j;$ and the dihedral angle $\theta_{ij}$ of the ideal hyperbolic tetrahedron at the edge $e_{ij}$ is the angle between the two totally geodesic planes respectively passing through the ideal vertices $\{v_i,v_j,v_k\}$ and $\{v_i,v_j,v_l\},$ where $\{k,l\}=\{1,2,3,4\}\setminus \{i,j\}.$ 

A \emph{decoration} of an ideal hyperbolic tetrahedron is an assignment of a horosphere $H_i,$  for each $i\in\{1,\dots, 4\},$ at the ideal vertex $v_i.$ We call an ideal hyperbolic tetrahedron together with a decoration a \emph{decorated ideal hyperbolic tetrahedron}. For a decorated ideal hyperbolic tetrahedron, the \emph{edge length} $x_{ij}$ of the edge $e_{ij}$ is defined to be the signed distance between the points of intersection $e_{ij}\cap H_i$ and $e_{ij}\cap H_j,$ with the sign being positive if $H_i$ and $H_j$ are disjoint, being $0$ if $H_i$ and $H_j$ intersect tangentially at a point on $e_{ij},$ and being negative if $H_i$ and $H_j$ intersect transversely. 

Let $\Delta$ be a decorated ideal hyperbolic tetrahedron  with the assigned horospheres $\{H_1,\dots,H_4\}$ and edge lengths $(x_{12},\dots,x_{34}).$ For $i\in\{1,2,3,4\},$ let $T_i=\Delta\cap H_i.$ Then all the $T_i$'s are Euclidean triangles similar to the Euclidean triangle $T$ with sides $\big(e^{\frac{x_{12}+x_{34}}{2}},e^{\frac{x_{13}+x_{24}}{2}},e^{\frac{x_{14}+x_{23}}{2}}\big).$ See Figure \ref{ideal}. As a consequence,  one has:
 \begin{enumerate}[(1)]
     \item  The dihedral angles of $\Delta$ satisfy
     \begin{equation}\label{sum=pi}
         \theta_{ij}+\theta_{ik}+\theta_{il}=\pi
     \end{equation}
     for $\{i,j,k,l\}=\{1,2,3,4\},$
     which implies that
     $$\theta_{12}=\theta_{34},\quad \theta_{13}=\theta_{24}\quad \text{and}\quad \theta_{14}=\theta_{23}. $$
 
     \item The edge lengths  of $\Delta$ satisfy the  triangle inequality 
     \begin{equation}\label{tri}
       e^{\frac{x_{ij}+x_{kl}}{2}}+e^{\frac{x_{ik}+x_{jl}}{2}}>e^{\frac{x_{il}+x_{jk}}{2}}  
     \end{equation}
     for $\{i,j,k,l\}=\{1,2,3,4\}.$
 \end{enumerate}

\begin{figure}[htbp]
\centering
\includegraphics[scale=0.3]{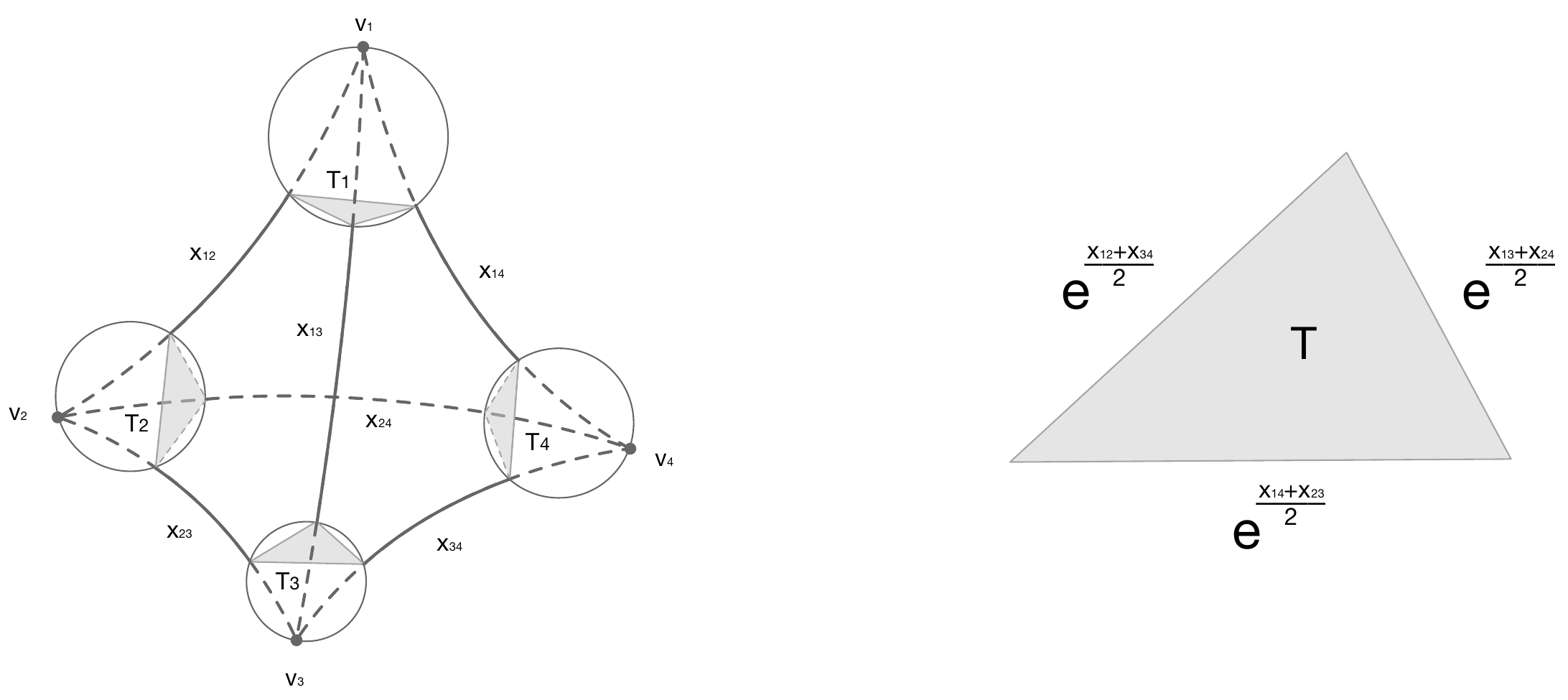}
\caption{Decorated ideal hyperbolic tetrahedron}
\label{ideal}
\end{figure}

 For a decorated ideal hyperbolic tetrahedron $\Delta$ with dihedral angle $(\theta_{12},\dots,\theta_{34})$ and edge lengths $\boldsymbol x= (x_{12},\dots,x_{34}),$ let $\mathrm{Vol}(\Delta)$ be the hyperbolic volume of $\Delta.$ Then the \emph{co-volume} of $\Delta$ is defined by
$$\mathrm{Cov}(\Delta)=\mathrm{Vol}(\Delta)+\frac{1}{2}\sum_{\{i,j\}\subset \{1,2,3,4\}}\theta_{ij}x_{ij},$$
which satisfies the following co-Schl\"afli Formula
\begin{equation}\label{co-sch}
    \frac{\partial \mathrm{Cov}(\Delta)}{\partial x_{ij}}=\frac{\theta_{ij}}{2}
\end{equation}
for each $\{i,j\}\subset\{1,2,3,4\}.$
 \medskip
 
For a decorated ideal hyperbolic tetrahedron $\Delta$ with edge lengths $\boldsymbol x= (x_{12},\dots,x_{34}),$  the  \emph{Gram matrix in the edge lengths} is defined by
\begin{equation*}
\mathrm{Gram}(\Delta)=\begin{bmatrix}
0 & -\frac{e^{x_{12}}}{2} & -\frac{e^{x_{13}}}{2}
    & -\frac{e^{x_{14}}}{2} \\
-\frac{e^{x_{12}}}{2} & 0 &-\frac{e^{x_{23}}}{2}
    & -\frac{e^{x_{24}}}{2}\\
   -\frac{e^{x_{13}}}{2} & -\frac{e^{x_{23}}}{2} & 0
    & -\frac{e^{x_{34}}}{2} \\
-\frac{e^{x_{14}}}{2} & -\frac{e^{x_{24}}}{2} &-\frac{e^{x_{34}}}{2}
    & 0
  \end{bmatrix}.
\end{equation*}
Then by a direct computation and (\ref{tri}), we have 
\begin{equation}\label{det-}
   \det\mathrm{Gram}(\Delta)<0.
\end{equation}
Also,  by the Law of Cosine to the Euclidean triangle $T,$
one has
\begin{equation}\label{cos}
    \cos\theta_{ij}=\frac{G_{kl}}{\sqrt{G_{kk}G_{ll}}},
\end{equation}
where $\{i,j,k,l\}=\{1,2,3,4\},$ and for $i,j\in\{1,2,3,4\},$ $G_{ij}$ is the $ij$-th co-factor of $\mathrm{Gram}(\Delta).$

\bigskip

 In the rest of this paper, we will fix and stick to the following identifications $(x_1,x_2,x_3,x_4,x_5,x_6)=(x_{12},x_{13},x_{23},x_{34},x_{24},x_{14})$  and 
$(\theta_1,\theta_2,\theta_3,\theta_4,\theta_5,\theta_6)=(\theta_{12},\theta_{13},\theta_{23},\theta_{34},\theta_{24},\theta_{14}).$

\subsection{Angle structures}\label{AS}

Let $(M,\mathcal{T})$ be an ideally triangulated 3-manifold with non-empty boundary, and let 
$E$ and $T$ respectively be the sets of edges and tetrahedra. We call a pair $(\Delta,e),$ $\Delta\in T$ and $e\in E,$  a corner of $\mathcal T$ if $\Delta$ is adjacent to  $e.$ We also denote by $\Delta\sim e$ when $(\Delta,e)$ is a corner. An \emph{angle assignment} on $(M, \mathcal T)$ assigns each corner a number $\theta_{(\Delta,e)}$ in $(0,\pi),$  called
the dihedral angle of $e$ in $\Delta,$ so that sum of
the dihedral angles at three edges in the same tetrahedron
$\Delta$ adjacent to each vertex equals $\pi.$ By (\ref{sum=pi}), these are exactly the conditions for six numbers in $(0,\pi)$ to be the dihedral angles of an ideal hyperbolic tetrahedron. 
 The {\it cone angle} of
an angle assignment is the assignment  $ \Theta \in \mathbb R^E$ that assigns  each edge $e$ to the sum of
dihedral angles at $e.$ An \emph{angle structure} is an  angle assignment with the cone angle $\Theta=(2\pi,\dots,2\pi).$


\section{Ideal $\mathrm {U}_{q\tilde q}\mathfrak{sl}(2;\mathbb R)$-$6j$ symbols}

\begin{proposition}[Well-definedness]\label{lem:ab-conver}
For any $(a_1,\ldots,a_6)\in\big(\frac{Q}{2}+\mathbf i \R\big)^6,$ the integral in
(\ref{b-6j}) converges absolutely for any vertical $\Gamma$ passing $(\frac{11Q}{6},2Q),$ and 
is independent of the choice of $\Gamma.$
\end{proposition}

\begin{proof}
As for $z\in\mathbb C,$ 
\begin{equation}\label{Tbnorm}
    |T_b(z)|=e^{\pi \big(\mathrm{Re}z -\frac{Q}{2}\big)\mathrm{Imz}},
\end{equation}
we have for $u\in \Gamma$ that 
\begin{equation}\label{Test}
\begin{split}
\Bigg|\frac{\prod_{i=1}^4 T_b(u-t_i)}{T_b(2Q-u)}\Bigg| = & e^{\sum_{i=1}^4\pi \big(\mathrm{Re}(u-t_i)-\frac{Q}{2}\big)\mathrm{Im}(u-t_i)-\pi\big(\mathrm{Re}(2Q-u)-\frac{Q}{2}\big)\mathrm{Im}(2Q-u)}\\
= & K e^{3\pi \big(\mathrm{Re}u-\frac{13Q}{6}\big)\mathrm{Im}u},
\end{split}
\end{equation}
where $K=e^{-\pi (\mathrm{Re}u-2Q)\sum_{i=1}^4\mathrm{Im}t_i}$ is a constant depending on $a_1,\dots, a_6, b$  and $\mathrm{Re}u,$ and is independent of $\mathrm{Im}u.$ Next, by \cite[Formula (2.20)]{LMSWY}, for $z\in \mathbb C$ with $0<\mathrm{Re}z <\frac{Q}{2},$ there is a constant $C$ depending on $b$ and $\mathrm{Re}z,$ and is independent of $\mathrm{Im}z,$ such that 
\begin{equation}\label{Sbbound}
|S_b(z)|\leqslant Ce^{\pi \big(\mathrm{Re}z -\frac{Q}{2}\big)|\mathrm{Imz}|}.
\end{equation}
As a consequence, we have 
\begin{equation*}
\begin{split}
\Bigg|\prod_{j=1}^3 S_b(q_j-u)\Bigg| \leqslant & C^3e^{\sum_{j=1}^3\pi \big(\mathrm{Re}(q_j-u)-\frac{Q}{2}\big)|\mathrm{Im}(q_j-u)|}= C^3e^{\pi \big(\frac{3Q}{2}-\mathrm{Re}u\big)\sum_{j=1}^3|\mathrm{Im}(q_j-u)|}.
\end{split}
\end{equation*}
Let $L$ be large enough so that $L>|\mathrm{Im}q_j|$ for all $j\in\{1,2,3\}.$ Then for $u\in\Gamma$ with $\mathrm{Im}u >L,$ we have 
\begin{equation}\label{Sest+}
\begin{split}
\Bigg|\prod_{j=1}^3 S_b(q_j-u)\Bigg| \leqslant &  C^3e^{-\pi \big(\frac{3Q}{2}-\mathrm{Re}u\big)\sum_{j=1}^3\mathrm{Im}(q_j-u)}=M_1e^{3\pi \big(\frac{3Q}{2}-\mathrm{Re}u\big)\mathrm{Im}u},
\end{split}
\end{equation}
where $M_1=C^3e^{-\pi\big(\frac{3Q}{2}-\mathrm{Re}u\big)\sum_{j=1}^3\mathrm{Im}q_j}$  is a constant depending on $a_1,\dots, a_6, b$  and $\mathrm{Re}u,$ and is independent of $\mathrm{Im}u;$
and for $u\in\Gamma$ with $\mathrm{Im}u <-L,$ we have 
\begin{equation}\label{Sest-}
\begin{split}
\Bigg|\prod_{j=1}^3 S_b(q_j-u)\Bigg|   \leqslant & C^3e^{\pi \big(\frac{3Q}{2}-\mathrm{Re}u\big)\sum_{j=1}^3\mathrm{Im}(q_j-u)}= M_2e^{-3\pi \big(\frac{3Q}{2}-\mathrm{Re}u\big)\mathrm{Im}u},
\end{split}
\end{equation}
where $M_2=C^3e^{\pi\big(\frac{3Q}{2}-\mathrm{Re}u\big)\sum_{j=1}^3\mathrm{Im}q_j}$  is a constant depending on $a_1,\dots, a_6, b$  and $\mathrm{Re}u,$ and is independent of $\mathrm{Im}u.$ Putting (\ref{Test}) and (\ref{Sest+}) together, we have
\begin{equation}\label{+bound}
    \Bigg|\frac{\prod_{i=1}^4 T_b(u-t_i)}{T_b(2Q-u)}\prod_{j=1}^3 S_b(q_j-u)\Bigg|\leqslant KM_1 e^{-2\pi Q\mathrm{Im}u}
\end{equation}
for $u\in \Gamma$ with $\mathrm{Im}u>L;$
and putting (\ref{Test}) and (\ref{Sest-}) together, we have
\begin{equation}\label{-bound}
    \Bigg|\frac{\prod_{i=1}^4 T_b(u-t_i)}{T_b(2Q-u)}\prod_{j=1}^3 S_b(q_j-u)\Bigg|\leqslant KM_2e^{6\pi \big(\mathrm{Re}u-\frac{11Q}{6}\big)\mathrm{Im}u}
\end{equation}
for $u\in \Gamma$ with $\mathrm{Im}u<-L.$
Therefore, for any vertical line $\Gamma$ passing the interval $\big(\frac{11Q}{6},2Q\big),$ the integrand in (\ref{b-6j}) decays exponentially at infinity, and the integral absolutely converges. 

Next, we show that the integral does not depend on the choice of the contour. For a vertical line $\Gamma$ and any $\lambda>0$, we let $\Gamma_\lambda=\{u\in \Gamma\mid -\lambda\leqslant \mathrm{Im}u\leqslant \lambda\}$; and to simplify the notation we denote the integrand of (\ref{b-6j}) by $F(u)$. Then for two vertical lines   $\Gamma_1$ and $\Gamma_2$ as described in the statement of the proposition, we have
$$\int_{\Gamma_1}F(u)d\mathrm{Im}u-\int_{\Gamma_2}F(u)d\mathrm{Im}u=-\mathbf i\lim_{\lambda\to\infty}\bigg(\int_{\Gamma_{1,\lambda}}F(u)du-\int_{\Gamma_{2,\lambda}}F(u)du\bigg),$$
which vanishes due to the analyticity of $F(u)$, the Residue Theorem and the estimates above.
\end{proof}

\begin{proposition}[Tetrahedral Symmetry]\label{tetrasym}
For any $(a_1,\dots,a_6)\in \big( \frac{Q}{2}+ \mathbf i \mathbb R\big)^6,$ 
\begin{align}
 \bigg|\begin{matrix} 
      a_1 & a_2 & a_3 \\
      a_4 & a_5 & a_6
\end{matrix} \bigg|_b= \bigg|\begin{matrix} 
      a_2 & a_1 & a_3 \\
      a_5 & a_4 & a_6
\end{matrix} \bigg|_b= \bigg|\begin{matrix} 
      a_1 & a_3 & a_2 \\
      a_4 & a_6 & a_5
 \end{matrix} \bigg|_b= \bigg|\begin{matrix}  
      a_1 & a_5 & a_6 \\
      a_4 & a_2 & a_3
 \end{matrix} \bigg|_b.
\end{align}
\end{proposition}

\begin{proof} As the above permutations of $a_k$'s do not change the $t_i$'s and $q_j$'s, the result follows directly from the definition.
\end{proof}

\begin{proposition}[Change of horospheres]\label{prop: horosphere}
For any $(a_1,\dots,a_6)\in \big( \frac{Q}{2}+ \mathbf i \mathbb R\big)^6,$ 
\begin{equation}\label{coh}
  \bigg|\begin{matrix}  a_1 & a_2 & a_3  \\ a_4 +\mathbf is & a_5+\mathbf is & a_6 +\mathbf is \end{matrix}  \bigg|_b= e^{-{\pi} Qs}\bigg|\begin{matrix}  a_1 & a_2 & a_3  \\ a_4 & a_5 & a_6 \end{matrix}\bigg|_b.  
\end{equation}
\end{proposition}

\begin{proof} As adding $\mathbf is$ to $a_4,$ $a_5$ and $a_6$ keeps $t_1$ unchanged, and changes $t_i$ to $t_i+2\mathbf is$ for $i\in\{2,3,4\}$ and $q_j$ to $q_j+2\mathbf is$ for $j\in\{1,2,3\},$ the result follows from the facts that 
\begin{equation}\label{TT}
    \frac{T_b(z+\mathbf i t)}{T_b(z)}=e^{\pi zt -\frac{\pi tQ}{2} +\frac{\pi\mathbf i t^2}{2}}\quad\text{and}\quad \frac{T_b(z-\mathbf i t)}{T_b(z)}=e^{-\pi zt +\frac{\pi tQ}{2} +\frac{\pi\mathbf i t^2}{2}}
\end{equation}
and the  change of variable $u\mapsto u+2\mathbf is$ in the integral (\ref{b-6j}).
\end{proof}

 \begin{proposition}[Pentagon Identity]\label{32}
For $(a_1,\dots,a_9)\in{\big(\frac{Q}{2}+\mathbf{i}\mathbb R\big)}^{9}$, the identity
\begin{equation}\label{pentagon}
  \begin{split}
\int_{ \frac{Q}{2}+\mathbf{i}\mathbb R}|T_b(2a)|^2\bigg|\begin{matrix} 
    a_1 & a_2 & a_6\\
      a_8 & a_7 & a 
   \end{matrix}\bigg|_b
   \bigg|\begin{matrix} 
    a_2 & a_3 & a_4\\
      a_9 & a_8 & a 
   \end{matrix}\bigg|_b
   \bigg|\begin{matrix} 
    a_3 & a_1 & a_5\\
      a_7 & a_9 & a 
   \end{matrix}\bigg|_b
   d\mathrm{Im}a=\bigg|\begin{matrix} 
    a_1 & a_2 & a_6\\
      a_4 & a_5 & a_3 
   \end{matrix}\bigg|_b
   \bigg|\begin{matrix} 
   a_7 & a_8 & a_6\\
      a_4 & a_5 & a_9 
   \end{matrix}\bigg|_b\\
   \end{split}
    \end{equation}
   holds as  continuous functions in $(a_1,\dots,a_9)\in {\big(\frac{Q}{2}+\mathbf{i}\mathbb R\big)}^9$.
\end{proposition}

The proof of Proposition \ref{32} relies on the following Lemma \ref{AI}, Lemma \ref{bib} and Corollary \ref{cor}.

\begin{lemma}[Idealization of $\mathrm {U}_{q\tilde q}\mathfrak{sl}(2;\mathbb R)$-$6j$ symbols]\label{AI} For any $(a_1,\dots,a_6)\in \big( \frac{Q}{2}+ \mathbf i \mathbb R\big)^6,$ 
\begin{equation}\label{6jlim}
 \bigg|\begin{matrix} a_1 & a_2 & a_3  \\ a_4 & a_5 & a_6 \end{matrix} \bigg|_b   = \lim_{s\in\mathbb R,\ s\to +\infty}  e^{2\pi Q s} \bigg\{\begin{matrix} a_1+\mathbf is & a_2 +\mathbf is & a_3+\mathbf is  \\ a_4+\mathbf is & a_5+\mathbf is & a_6+\mathbf is \end{matrix} \bigg\}_b.
\end{equation}
Here  $\{\dots\}_b$ is the $\mathrm {U}_{q\tilde q}\mathfrak{sl}(2;\mathbb R)$-$6j$ symbol  computed for any  $(a_1,\dots,a_6)\in \big( \frac{Q}{2}+ \mathbf i \mathbb R_{>0}\big)^6$ by 
\begin{equation}\label{b6}
\bigg\{\begin{matrix} a_1 & a_2 & a_3 \\ a_4 & a_5 & a_6 \end{matrix} \bigg\}_b=\Bigg(\frac{1}{\prod_{i=1}^4\prod_{j=1}^4S_b(q_j-t_i)}\Bigg)^{\frac{1}{2}}\int_\Gamma \prod_{i=1}^4S_b(u-t_i)\prod_{j=1}^4S_b(q_j-u)d\mathrm{Im} u,
\end{equation}
where $\Gamma$ is any vertical line passing the interval $\big(\frac{3Q}{2},2Q\big)$ oriented upward.  
\end{lemma}

\begin{proof} For $s\in\mathbb R,$ we let $t_{i,s}=t_i+ 3\mathbf is$ for each $i\in\{1,2,3,4\},$ and let $q_{j,s}=q_j+4\mathbf is$ for each $j\in\{1,2,3\}$ and let $q_{4,s}=q_4=2Q.$ Fix a vertical line $\Gamma$ passing the interval $\big(\frac{11Q}{6},2Q\big)\subset\big(\frac{3Q}{2},2Q\big),$ and for each $u\in \Gamma$ and $s\in\mathbb R,$ let 
$u_s=u+4\mathbf i s.$ We define
$$F_s(u)\doteq e^{2\pi Qs}\Bigg(\frac{1}{\prod_{i=1}^4\prod_{j=1}^4S_b(q_{j,s}-t_{i,s})}\Bigg)^{\frac{1}{2}} \prod_{i=1}^4S_b(u_s-t_{i,s})\prod_{j=1}^4S_b(q_{j,s}-u_s)$$
and 
$$F(u)\doteq \Bigg(\frac{\prod_{i=1}^4 T_b(2Q-t_i)}{\prod_{i=1}^4\prod _{j=1}^3 T_b(q_j-t_i)}\Bigg)^{\frac{1}{2}} \frac{\prod_{i=1}^4 T_b(u-t_i)}{T_b(2Q-u)}\prod_{j=1}^3S_b(q_j-u).$$
Then by (\ref{b6}) and the change of variable $u\mapsto u_s,$ we have 
$$e^{2\pi Q s} \bigg\{\begin{matrix} a_1+\mathbf is & a_2 +\mathbf is & a_3+\mathbf is  \\ a_4+\mathbf is & a_5+\mathbf is & a_6+\mathbf is \end{matrix} \bigg\}_b=\int_\Gamma F_s(u)d\mathrm{Im}u;$$
and by (\ref{b-6j}), we have
$$ \bigg|\begin{matrix} a_1 & a_2 & a_3  \\ a_4 & a_5 & a_6 \end{matrix} \bigg|_b=\int_\Gamma F(u)d\mathrm{Im}u.$$
We claim that:
\begin{enumerate}[(1)]
    \item For each $u\in\Gamma,$
$$\lim_{s\to +\infty} F_s(u) = F(u).$$
    \item There is an integrable function $G$ on $\Gamma$ such that for each $u\in \Gamma$ and $s>0$ sufficiently large, 
    $$|F_s(u)|\leqslant |G(u)|.$$
\end{enumerate}
Then the result follows immediately from (1), (2) and the Lebesgue Dominated Convergence Theorem. 
\medskip

To see (1), by \cite[Formula (B.53)]{E}, we have for $z\in\mathbb C$ with $0<\mathrm{Re}z<Q,$
\begin{equation}\label{ST}
\lim_{\mathrm{Im}z\to +\infty} \frac{S_b(z)}{T_b(z)}=1\quad\text{and}\quad\lim_{\mathrm{Im}z\to -\infty} {S_b(z)}{T_b(z)}=1.
\end{equation}
Then we have 
\begin{equation*}
\begin{split}
  \lim_{s\to +\infty} F_s(u)= & \lim_{s\to +\infty} e^{2\pi Qs} \Bigg(\frac{\prod_{i=1}^4 T_b(2Q-t_i-3\mathbf is)}{\prod_{i=1}^4\prod _{j=1}^3 T_b(q_j-t_i+\mathbf is)}\Bigg)^{\frac{1}{2}} \frac{\prod_{i=1}^4 T_b(u-t_i+\mathbf is)}{T_b(2Q-u-4\mathbf is)}\prod_{j=1}^3S_b(q_j-u) \\
=& F(u),
\end{split}
\end{equation*}
where the last equality comes from (\ref{TT}). This completes the proof of (1).
\medskip

To see (2), by (\ref{Tbnorm}) and (\ref{Sbbound}) with the constant $C$ therein, we have
\begin{equation}\label{STC}
    |S_b(z)|\leqslant C |T_b(z)|
\end{equation}
if $\mathrm{Im}z\geqslant 0,$
and 
$$|S_b(z)|\leqslant \frac{C}{ |T_b(z)|}$$
if $\mathrm{Im}z < 0.$ Together with \cite[Formula (2.20)]{LMSWY} that 
\begin{equation}\label{=1}
    \bigg|\prod_{i=1}^4\prod_{j=1}^4S_b(q_j-t_i)^{-1}\bigg|=1,
\end{equation}
 we have
\begin{equation*}
\begin{split}
|F_s(u)|\leqslant & C^5 e^{2\pi Qs}  \Bigg|\frac{\prod_{i=1}^4 T_b(u-t_i+\mathbf is)}{T_b(2Q-u-4\mathbf is)}\prod_{j=1}^3 S_b(q_j-u)\Bigg|\\
= &  C^5  \Bigg|\frac{\prod_{i=1}^4 T_b(u-t_i)}{T_b(2Q-u)}\prod_{j=1}^3 S_b(q_j-u)\Bigg|,
\end{split}    
\end{equation*}
where the equality comes from (\ref{TT}). Now by (\ref{+bound}), (\ref{-bound}) and the constants $L,$ $K,$ $M_1$ and $M_2$ therein, we have 
$$|F_s(u)| \leqslant C^5KM_1 e^{-2\pi Q\mathrm{Im}u}$$
for $u\in \Gamma$ with $\mathrm{Im}u>L,$ and 
$$|F_s(u)| \leqslant C^5KM_2 e^{6\pi \big(\mathrm{Re}u-\frac{11Q}{6}\big)\mathrm{Im}u}$$
for $u\in \Gamma$ with $\mathrm{Im}u<-L.$
Also, by (1), there is a $P>0$ such that 
$$|F_s(u)|<P|F(u)|$$
for $s$ sufficiently large and for each $u\in\Gamma$ with $\mathrm{Im}u \in [-L, L].$ Then we can defined the desired function $G$ on $\Gamma$ by 
$$G(u) \doteq  Ne^{-2\pi Q\mathrm{Im}u}$$
if  $\mathrm{Im}u\geqslant 0,$ and 
$$G(u) \doteq  Ne^{6\pi \big(\mathrm{Re}u-\frac{11Q}{6}\big)\mathrm{Im}u}$$
if $\mathrm{Im}u\leqslant 0,$
for some large enough real number 
$$N>\max \bigg\{\sup _{\mathrm{Im}u\in [0,L]}\Big|PF(u)e^{2\pi Q\mathrm{Im}u}\Big|,\sup _{\mathrm{Im}u\in [-L,0]}\Big|PF(u)e^{-6\pi \big(\mathrm{Re}u-\frac{11Q}{6}\big)\mathrm{Im}u}\Big|, C^5KM_1, C^5KM_2 \bigg\}.$$
This completes the proof of (2).
\end{proof}

\begin{lemma}\label{bib}
 For $(a_1,\dots,a_6)\in \big(\frac{Q}{2}+\mathbf i\mathbb R\big)^6,$ let 
 $m = \frac{1}{6}\sum_{k=1}^6\mathrm{Im}a_k;$
 and for each $j\in\{1,2,3\},$ let 
 $b_j=\mathrm{Im}a_j+\mathrm{Im}a_{j+3}-2m.$ 
 Then for each $c \in \big(\frac{11Q}{6},2Q\big),$ there is a constant $C_{b,c}>0$ depending only on $b$ and $c$ such that:
 \begin{enumerate}[(1)]      
 \item For all $(a_1,\dots,a_6)\in \big(\frac{Q}{2}+\mathbf i\mathbb R)^6$ and  $s>0$ large enough so that $\mathrm{Im}a_k+s>0$ for each $k\in\{1,\dots,6\},$ 
 $$\Bigg|e^{2\pi Q s} \bigg\{\begin{matrix} a_1+\mathbf is & a_2 +\mathbf is & a_3+\mathbf is  \\ a_4+\mathbf is & a_5+\mathbf is & a_6+\mathbf is \end{matrix} \bigg\}_b\Bigg|\leqslant C_{b,c} e^{-2\pi Qm + \pi (11Q-6c)\max\{b_1,b_2,b_3\}}.$$

 \item For all $(a_1,\dots,a_6)\in \big(\frac{Q}{2}+\mathbf i\mathbb R)^6,$ 
 $$\Bigg|\bigg|\begin{matrix} a_1 & a_2 & a_3  \\ a_4 & a_5 & a_6 \end{matrix} \bigg|_b\Bigg|\leqslant C_{b,c} e^{-2\pi Qm +\pi  (11Q-6c) \max\{b_1,b_2,b_3\}}.$$
 \end{enumerate}
\end{lemma}

\begin{proof} By Lemma \ref{AI}, (2) is a direct consequence of (1), hence it suffices to prove (1).
\medskip

To this end, recall the notations from the beginning of the proof of Lemma \ref{AI} that for  $s\in\mathbb R,$ $t_{i,s}=t_i+ 3\mathbf is$ for each $i\in\{1,2,3,4\},$ $q_{j,s}=q_j+4\mathbf is$ for each $j\in\{1,2,3\}$ and $q_{4,s}=q_4=2Q.$ Let  $\Gamma_c=\{ u\in \mathbb C\ |\ \mathrm{Re}u =c\},$ which is a vertical line passing the interval $\big(\frac{11Q}{6},2Q\big)\subset\big(\frac{3Q}{2},2Q\big);$ and for each $u\in \Gamma_c$ and $s\in\mathbb R,$ let 
$u_s=u+4\mathbf i s.$ Define
$$F_s(u)\doteq e^{2\pi Qs}\Bigg(\frac{1}{\prod_{i=1}^4\prod_{j=1}^4S_b(q_{j,s}-t_{i,s})}\Bigg)^{\frac{1}{2}} \prod_{i=1}^4S_b(u_s-t_{i,s})\prod_{j=1}^4S_b(q_{j,s}-u_s)$$
so that 
\begin{equation}\label{Fsint}
    e^{2\pi Q s} \bigg\{\begin{matrix} a_1+\mathbf is & a_2 +\mathbf is & a_3+\mathbf is  \\ a_4+\mathbf is & a_5+\mathbf is & a_6+\mathbf is \end{matrix} \bigg\}_b=\int_{\Gamma_c} F_s(u)d\mathrm{Im}u.
\end{equation}

Now by (\ref{=1}) and (\ref{Sbbound}) with the constant $C$ therein and doing the change of variable $x=\mathrm{Im}u-4m,$ we have 
\begin{equation}\label{Fbound}
\begin{split}
|F_s(u)|\leqslant &C^8 e^{2\pi Qs -\pi(2Q-c)\sum_{i=1}^4|\mathrm{Im}(u_s-t_{i,s})|-\pi \big(c-\frac{3Q}{2}\big)\sum_{j=1}^4|\mathrm{Im}(u_s-q_{j,s})|}\\
\leqslant & C^8 e^{2\pi Qs-\pi(2Q-c)\sum_{i=1}^4 \mathrm{Im}(u_s-t_{i,s})-\pi \big(c-\frac{3Q}{2}\big)\sum_{j=1}^3|\mathrm{Im}(u_s-q_{j,s})|-\pi \big(c-\frac{3Q}{2}\big)\mathrm{Im}u_s}\\
=& C^8 e^{\pi(2Q-c)\sum_{i=1}^4\mathrm{Im}t_i-\pi\big(\frac{13Q}{2}-3c\big)\mathrm{Im}u -\pi\big(c-\frac{3Q}{2}\big)\sum_{j=1}^3|\mathrm{Im}(u-q_j)|}\\
=& C^8 e^{-2\pi Qm-\pi\big(\frac{13Q}{2}-3c\big)x -\pi\big(c-\frac{3Q}{2}\big)\sum_{j=1}^3|x+b_j|}.
\end{split}
\end{equation}

Next, let
$$f(x)\doteq -2\pi Qm-\pi\bigg(\frac{13Q}{2}-3c\bigg)x -\pi\bigg(c-\frac{3Q}{2}\bigg)\sum_{j=1}^3\big|x+b_j\big|$$
be the exponent above. 
Then it is a piecewise linear function in $x$ with the set of  singular points $\{-b_1,-b_2,-b_3\},$ and with the slopes of the pieces ordered increasingly in $x$ respectively $-11Q+6c>0,$ $-8Q+4c<0,$ $-5Q+2c<0$ and $-2Q<0.$ Therefore, letting $b_\text{max}=\max\{b_1,b_2,b_3\},$ $f(x)$ achieves the maximum  at $x=\min\{-b_1,-b_2,-b_3\}=-b_\text{max},$  with the value
\begin{equation}\label{fmax}
\begin{split}
    f(-b_\text{max})=&-2\pi Qm+\pi\bigg(\frac{13Q}{2}-3c\bigg)b_\text{max} -\pi\bigg(c-\frac{3Q}{2}\bigg)\sum_{b\in \{b_1,b_2,b_3\}\setminus\{b_\text{max}\}}(b_\text{max}-b)\\
    = & -2\pi Q m + \pi (11Q-6c) b_\text{max},
\end{split}
\end{equation}
where the last equality comes from that $b_1+b_2+b_3 =0.$

Finally, putting (\ref{Fsint}), (\ref{Fbound}) and (\ref{fmax}) together, we have 
\begin{equation*}
\begin{split}
  \Bigg|e^{2\pi Q s} \bigg\{\begin{matrix} a_1+\mathbf is & a_2 +\mathbf is & a_3+\mathbf is  \\ a_4+\mathbf is & a_5+\mathbf is & a_6+\mathbf is \end{matrix} \bigg\}_b\Bigg| \leqslant & \int_{\Gamma_c} |F_s(u)|d\mathrm{Im}u\\
\leqslant  & \int_{\mathbb R} C^8 e^{f(x)}dx =C_{b,c} e^{-2\pi Qm + \pi (11Q-6c)\max\{b_1,b_2,b_3\}}
\end{split}    
\end{equation*}
for a constant  $C_{b,c}$  depending only on $b$ and $c,$ and the equality comes from  the concrete integration  on each of the linear pieces of $f(x).$
\end{proof}

As a consequence of Lemma \ref{bib} (1), we have the following Corollary \ref{cor}, which will be needed in the proof of Proposition \ref{32}.

\begin{corollary}\label{cor}
For a fixed $(a_2,\dots,a_6)\in\big(\frac{Q}{2}+\mathbf i\mathbb R\big)^5$ and $c\in \big(\frac{11Q}{6},2Q\big),$ there is a $D=D_{b,c,a_2,\dots,a_6}>0$ depending only on $b,$ $c$ and $(a_2,\dots,a_6)$ such that:
\begin{enumerate}[(1)]
    \item For all $a\in \frac{Q}{2}+\mathbf i\mathbb R$ with $\mathrm{Im}a\geqslant 0$ and $s>0$ large enough so that $\mathrm{Im}a_k + s>0$ for each $k\in\{2,\dots, 6\}$ and $\mathrm{Im}a + s>0,$
$$\Bigg|e^{2\pi Q s} \bigg\{\begin{matrix} a+\mathbf is & a_2 +\mathbf is & a_3+\mathbf is  \\ a_4+\mathbf is & a_5+\mathbf is & a_6+\mathbf is \end{matrix} \bigg\}_b\Bigg|\leqslant De^{\pi(7Q-4c)\mathrm{Im}a}.$$

\item For all $a\in \frac{Q}{2}+\mathbf i\mathbb R$ with $\mathrm{Im}a\leqslant 0$ and $s>0$ large enough so that $\mathrm{Im}a_k + s>0$ for each $k\in\{2,\dots, 6\}$ and $\mathrm{Im}a + s>0,$
$$\Bigg|e^{2\pi Q s} \bigg\{\begin{matrix} a+\mathbf is & a_2 +\mathbf is & a_3+\mathbf is  \\ a_4+\mathbf is & a_5+\mathbf is & a_6+\mathbf is \end{matrix} \bigg\}_b\Bigg|\leqslant De^{\pi(-4Q+2c)\mathrm{Im}a}.$$
\end{enumerate}
\end{corollary}

\begin{proof}
From the definition, we have 
$$m=\frac{\mathrm{Im}a}{6}+K_1\quad\text{and}\quad b_1=\frac{2\mathrm{Im}a}{3}+K_2,$$
where $K_1$ and $K_2$ are certain linear combinations of $\mathrm{Im}a_1,\dots,\mathrm{Im}a_5.$

For (1) that $\mathrm{Im}a\geqslant 0,$ as $11Q-6c<0$ and  $\max\{b_1,b_2,b_3\}\geqslant b_1,$ we have
\begin{equation*}
\begin{split}
    -2\pi Qm + \pi &(11Q-6c)\max\{b_1,b_2,b_3\} \\
\leqslant &-2\pi Q\bigg(\frac{\mathrm{Im}a}{6}+K_1\bigg) + \pi (11Q-6c) \bigg(\frac{2\mathrm{Im}a}{3}+K_2\bigg)\\
=& \pi(7Q-4c)\mathrm{Im}a -2\pi Q K_1 +\pi(11Q-6c)K_2,
\end{split}
\end{equation*}
and the result follows from Lemma \ref{bib} with $D=C_{b,c}e^{-2\pi Q K_1 + \pi(11Q-6c)K_2}.$

For (2) that $\mathrm{Im}a\leqslant 0,$ as $b_1+b_2+b_3=0,$ we have $\max\{b_1,b_2,b_3\}\geqslant \frac{b_2+b_3}{2}=-\frac{b_1}{2}.$ As $11Q-6c<0,$ we have
\begin{equation*}
\begin{split}
    -2\pi Qm + \pi &(11Q-6c)\max\{b_1,b_2,b_3\} \\
\leqslant &-2\pi Q\bigg(\frac{\mathrm{Im}a}{6}+K_1\bigg) -\frac{\pi}{2} (11Q-6c) \bigg(\frac{2\mathrm{Im}a}{3}+K_2\bigg)\\
=& \pi(-4Q+2c)\mathrm{Im}a -2\pi Q K_1 - \frac{\pi}{2}(11Q-6c)K_2,
\end{split}
\end{equation*}
and the result follows from Lemma \ref{bib} with $D=C_{b,c}e^{-2\pi QK_1 - \frac{\pi}{2}(11Q-6c)K_2}.$
\end{proof}

\begin{proof}[Proof of Proposition \ref{32}]
For $s\in \mathbb R$ and $k\in \{1,\dots, 9\},$ let $a_{k,s}=a_k+\mathbf is.$ Then by \cite[Proposition 6.1 (1)]{LMSWY}, for each $s\in \mathbb R$ such that $\mathrm{Im}a_{k,s}>0$ for each $k\in\{1,\dots,9\},$
 we have 
\begin{equation}\label{5gon}
\begin{split}
e^{4\pi Qs}\int_{ \frac{Q}{2}+\mathbf{i}\mathbb R_{>0}}|S_b(2a)|^2\bigg\{\begin{matrix} 
    a_{1,s} & a_{2,s} & a_{6,s}\\
      a_{8,s} & a_{7,s} & a 
   \end{matrix}\bigg\}_b
  & \bigg\{\begin{matrix} 
    a_{2,s} & a_{3,s} & a_{4,s}\\
      a_{9,s} & a_{8,s} & a 
   \end{matrix}\bigg\}_b
 \bigg\{\begin{matrix} 
    a_{3,s} & a_{1,s} & a_{5,s}\\
      a_{7,s} & a_{9,s} & a 
   \end{matrix}\bigg\}_b
   d\mathrm{Im}a\\
   =  e^{4\pi Qs} & \bigg\{\begin{matrix} 
    a_{1,s} & a_{2,s} & a_{6,s}\\
      a_{4,s} & a_{5,s} & a_{3,s} 
   \end{matrix}\bigg\}_b
   \bigg\{\begin{matrix} 
   a_{7,s} & a_{8,s} & a_{6,s}\\
      a_{4,s} & a_{5,s} & a_{9,s} 
   \end{matrix}\bigg\}_b.   
\end{split}
\end{equation}
We will show that, as $s\to +\infty,$ the two sides of (\ref{5gon}) respectively converge to the two sides of (\ref{pentagon}), hence (\ref{pentagon}) holds. 
\medskip

For the left hand sides, for each $s\in \mathbb R$ such that $\mathrm{Im}a_{k,s}>0$ for all $k\in\{1,\dots,9\},$ let $\mathcal F_s$ be the function on $\frac{Q}{2}+\mathbf i\mathbb R$ given by
$$\mathcal F_s(a)\doteq \chi_{(-s,+\infty)}(\mathrm{Im}a)e^{4\pi Q s}|S_b(2a_s)|^2
\bigg\{\begin{matrix} 
    a_{1,s} & a_{2,s} & a_{6,s}\\
      a_{8,s} & a_{7,s} & a_s 
   \end{matrix}\bigg\}_b
   \bigg\{\begin{matrix} 
    a_{2,s} & a_{3,s} & a_{4,s}\\
      a_{9,s} & a_{8,s} & a_s 
   \end{matrix}\bigg\}_b
 \bigg\{\begin{matrix} 
    a_{3,s} & a_{1,s} & a_{5,s}\\
      a_{7,s} & a_{9,s} & a_s 
   \end{matrix}\bigg\}_b,$$
   where $a_s=a+\mathbf is$ and $\chi_{(-s,+\infty)}$ is the characteristic function of the interval $(-s,+\infty)$ that takes value $1$ for $x\in (-s,+\infty),$ and takes value $0$ otherwise;
   and let 
   $$\mathcal F(a)\doteq |T_b(2a)|^2\bigg|\begin{matrix} 
    a_1 & a_2 & a_6\\
      a_8 & a_7 & a 
   \end{matrix}\bigg|_b
   \bigg|\begin{matrix} 
    a_2 & a_3 & a_4\\
      a_9 & a_8 & a 
   \end{matrix}\bigg|_b
   \bigg|\begin{matrix} 
    a_3 & a_1 & a_5\\
      a_7 & a_9 & a 
   \end{matrix}\bigg|_b.$$
Then by the change of variable $a\mapsto a_s,$ the left hand side of (\ref{5gon}) equals 
   $$\int_{\frac{Q}{2}+\mathbf i\mathbb R} \mathcal F_s(a) d\mathrm{Im}a;$$
   and the left hand side of (\ref{pentagon}) equals 
      $$\int_{\frac{Q}{2}+\mathbf i\mathbb R} \mathcal F(a) d\mathrm{Im}a.$$
We will show that:
\begin{enumerate}[(1)]
    \item For each $a \in \frac{Q}{2}+\mathbf i\mathbb R,$
    $$\lim_{s\to +\infty}\mathcal F_s(a)=\mathcal F(a).$$

    \item There is an integrable function $\mathcal G$ on $\frac{Q}{2}+\mathbf i \mathbb R$ such that 
    $$|\mathcal F_s(a)|\leqslant \mathcal G(a)$$
    for all $a\in \frac{Q}{2}+\mathbf i \mathbb R.$
\end{enumerate}
Then by the Lebesgue Dominated Convergence Theorem, the left hand side of (\ref{5gon}) converges to the left hand side of (\ref{pentagon}) as $s\to +\infty.$
\smallskip

To show (1), as $T_b(2a_s)=T_b(2a)e^{2\pi Qs},$ we have
\begin{equation}\label{FsB}
    \begin{split}
        \mathcal F_s(a)= &\chi_{(-s,+\infty)}(\mathrm{Im}a)\Bigg|\frac{S_b(2a_s)}{T_b(2a_s)}\Bigg|^2\\
&|T_b(2a)|^2
e^{6\pi Q s} \bigg\{\begin{matrix} 
    a_{1,s} & a_{2,s} & a_{6,s}\\
      a_{8,s} & a_{7,s} & a_s 
   \end{matrix}\bigg\}_b
   \bigg\{\begin{matrix} 
    a_{2,s} & a_{3,s} & a_{4,s}\\
      a_{9,s} & a_{8,s} & a_s 
   \end{matrix}\bigg\}_b
 \bigg\{\begin{matrix} 
    a_{3,s} & a_{1,s} & a_{5,s}\\
      a_{7,s} & a_{9,s} & a_s 
   \end{matrix}\bigg\}_b.
    \end{split}
\end{equation}
Then (1) follows from (\ref{ST}), Lemma \ref{AI} and that 
$$\lim_{s\to +\infty}\chi_{(-s,+\infty)}(\mathrm{Im}a)=1.$$

To show (2), fix a $c\in \big(\frac{23Q}{12},2Q)\subset\big(\frac{11Q}{6},2Q\big),$ and let $C$ and $D$ respectively be the constants in (\ref{Sbbound}) and in Corollary \ref{cor}. We define the function $\mathcal G$ on $\frac{Q}{2}+\mathbf i\mathbb R$ as follows: If $\mathrm{Im}a\geqslant 0,$ then let
$$\mathcal G(a) \doteq CD e^{\pi (23Q-12c)\mathrm{Im}a};$$
and if $\mathrm{Im}a\leqslant 0,$ then let
$$\mathcal G(a) \doteq  CD e^{\pi (-10Q+6c)\mathrm{Im}a}.$$
As $23Q-12c<0$ and $-10Q+6c>0,$ $\mathcal G$ is integrable on $\frac{Q}{2}+\mathbf i\mathbb R;$ and by (\ref{FsB}), (\ref{STC}) and Corollary (\ref{cor}),  $|\mathcal F_s(a)|\leqslant \mathcal G(a)$ for all $a\in\frac{Q}{2}+\mathbf i\mathbb R.$ This completes the proof of (2).
\medskip

On the other hand, as a direct consequence of Lemma \ref{AI}, the right hand side of (\ref{5gon}) converges to the right hand side of (\ref{pentagon}) as $s\to +\infty.$ This completes the proof. 
\end{proof}

As a consequence of Lemma \ref{bib} (2), we have the following Corollary  \ref{thetabj}, which will be needed in the proof of Theorem \ref{converge}.

\begin{corollary}\label{thetabj}
Let $\boldsymbol \theta = (\theta_1,\dots,\theta_6)$ be the dihedral angles of an ideal hyperbolic tetrahedron. Then there exist constants $C_{b,\boldsymbol \theta}>0$ and $\epsilon_{b,\boldsymbol \theta}>0$ depending only on $b$ and $\boldsymbol \theta$ such that for all $\boldsymbol a=(a_1,\dots,a_6)\in \big(\frac{Q}{2}+\mathbf i\mathbb R)^6,$ 
 $$\Bigg|\bigg|\begin{matrix} a_1 & a_2 & a_3  \\ a_4 & a_5 & a_6 \end{matrix} \bigg|_b\Bigg|\leqslant C_{b,\boldsymbol \theta} e^{- Q (\boldsymbol \theta \cdot \mathrm{Im}\boldsymbol a)  - \epsilon_{b,\boldsymbol \theta}\max\{b_1,b_2,b_3\}},$$
 where $\mathrm{Im}\boldsymbol a=(\mathrm{Im}a_1,\dots,\mathrm{Im}a_6),$ and $b_j=\mathrm{Im}a_j+\mathrm{Im}a_{j+3}-\frac{1}{3}\sum_{k=1}^6\mathrm{Im}a_k$  for each $j\in\{1,2,3\}.$ 
\end{corollary}

\begin{proof}
Let us first recall that as dihedral angles of an ideal hyperbolic tetrahedron, we have $\theta_1=\theta_4,$ $\theta_2=\theta_5,$ $\theta_3=\theta_6$ and $\theta_1+\theta_2+\theta_3=\pi.$ Let $$\theta_\text{min} =\min\{\theta_1,\theta_2,\theta_3\}.$$
As $b_1+b_2+b_3=0,$ we have 
\begin{equation}\label{tbj}
\sum_{j=1}^3\theta_jb_j= \sum_{j=1}^3(\theta_j-\theta_\text{min})b_j \leqslant (\pi -3\theta_\text{min})\max\{b_1,b_2,b_3\}.
\end{equation}

Now, for $\delta>0$ sufficiently small, let $c=\big(2-\frac{\theta_\text{min}-\delta}{2\pi}\big)Q.$ Then $c\in \big(\frac{11Q}{6},2Q\big)$ as $\theta_\text{min}\leqslant \frac{\pi}{3},$  and by  Lemma \ref{bib} (2) with  the constant $C_{b,c}$ therein, we have 
\begin{equation*}
\begin{split}
    \Bigg|\bigg|\begin{matrix} a_1 & a_2 & a_3  \\ a_4 & a_5 & a_6 \end{matrix} \bigg|_b\Bigg|\leqslant &  C_{b,c} e^{-\frac{\pi}{3} Q\sum_{k=1}^6\mathrm{Im}a_k -(\pi-3\theta_\text{min}+3\delta) Q\max\{b_1,b_2,b_3\}}\\
\leqslant &  C_{b,c} e^{-\frac{\pi}{3} Q\sum_{k=1}^6\mathrm{Im}a_k -Q\sum_{j=1}^3\theta_jb_j-3\delta Q\max\{b_1,b_2,b_3\}}\\
= &  C_{b,c} e^{-Q\sum_{k=1}^6\theta_k\mathrm{Im}a_k-3\delta Q\max\{b_1,b_2,b_3\}},
\end{split}    
\end{equation*}
where the second inequality comes from (\ref{tbj}), and the result follows with $C_{b,\boldsymbol \theta}=C_{b,c}$ and $\epsilon_{b,\boldsymbol \theta}=-3\delta Q.$
\end{proof}


\section{Asymptotics  of ideal $\mathrm {U}_{q\tilde q}\mathfrak{sl}(2;\mathbb R)$-$6j$ symbols}

The goal of this section is to prove Theorem \ref{6jasymp}. 
 The main tool in the proof of Theorem \ref{6jasymp} is the following Saddle Point Approximation from \cite[Proposition 6.1]{WY} (see also \cite[Proposition 2.22 and Remark 2.23]{LMSWY}). 

\begin{proposition}\label{saddle}
Let $D$ be a region in $\mathbb C^n$ and let $f(z_1,\dots, z_n)$ and $g(z_1,\dots, z_n)$ be holomorphic functions on $D$ independent of $\hbar$. Let $f_\hbar(z_1,\dots,z_n)$ be a holomorphic function of the form
$$ f_\hbar(z_1,\dots, z_n) = f(z_1,\dots, z_n) + \upsilon_\hbar(z_1,\dots,z_n)\hbar^2.$$
Let $S$ be an embedded real $n$-dimensional closed disk in $D$ and let $(c_1,\dots, c_n)$ be a point on $S.$ If
\begin{enumerate}[(i)]
\item $(c_1, \dots, c_n)$ is a critical point of $f$ in $D,$
\item $\mathrm{Re}(f)(c_1,\dots,c_n) > \mathrm{Re}(f)(z_1,\dots,z_n)$ for all $(z_1,\dots,z_n) \in S\setminus \{(c_1,\dots,c_n)\},$
\item the Hessian matrix $\mathrm{Hess}(f)(c_1,\dots,c_n)$ of $f$ at $(c_1,\dots,c_n)$ is non-singular,
\item $g(c_1,\dots,c_n) \neq 0,$  
\item $|\upsilon_\hbar(z_1,\dots,z_n)|$ is bounded from above by a constant independent of $\hbar$ in $D,$ and
\item $S$ is smoothly embedded around $(c_1,\dots,c_n),$ 
\end{enumerate}
then
\begin{equation*}
\begin{split}
 \int_S g(z_1,\dots, z_n) &e^{\frac{f_\hbar(z_1,\dots,z_n)}{\hbar}} dz_1\dots dz_n\\
 &= \Big(2\pi \hbar\Big)^{\frac{n}{2}}\frac{g(c_1,\dots,c_n)}{\sqrt{\det\big(-\mathrm{Hess}(f)(c_1,\dots,c_n)\big)}} e^{\frac{f(c_1,\dots,c_n)}{\hbar}} \Big( 1 + O \big(\hbar\big)\Big).
 \end{split}
 \end{equation*}
\end{proposition}

To apply Proposition \ref{saddle}, we will study the propertises of the  relevent functions in Subsection \ref{3.2}. Based on these preparations, we will prove Theorem \ref{6jasymp} in Subsection \ref{3.3}.

\subsection{Properties of relevant functions}\label{3.2}

For $(x_1,\dots,x_6)\in\mathbb R^6,$ let 
$$y_1=\frac{x_1+x_4}{2},\quad y_2=\frac{x_2+x_5}{2}\quad\text{and}\quad y_3=\frac{x_3+x_6}{2};$$
and for each  $k\in\{1,\dots,6\},$ let 
$$a_k = \frac{Q}{2} + \mathbf i \frac{x_k}{2\pi b}.$$ 
Then by the change of variable 
$$\zeta =\pi b(2Q-u)+\mathbf i\sum_{i=1}^3y_i$$
and a direct computation, we have
\begin{equation}\label{idealb6j}
    \bigg|\begin{matrix} a_1 & a_2 & a_3 \\ a_4 & a_5 & a_6 \end{matrix} \bigg|_b= \frac{1}{\pi b}\int_\Gamma \exp\bigg(\frac{V_{\boldsymbol x,b}(\zeta)}{2\pi\mathbf i b^2}\bigg)d\zeta,
\end{equation}
where  $\Gamma$ is any vertical line passing the interval $(0,\frac{\pi}{6}),$ and 
\begin{equation}\label{Vxb}
\begin{split}
    V_{\boldsymbol x,b}(\zeta)=&\sum_{i=1}^3y_i^2-2\sum_{1\leqslant i< j\leqslant 3}y_iy_j-\pi\mathbf i \big(1+b^2\big)\sum_{i=1}^3y_i-\frac{\pi^2}{6}\big(1-3b^2+b^4\big)\\
    & + 3\zeta^2+\bigg(\pi\big(1+b^2)-2\mathbf i \sum_{i=1}^3y_i\bigg)\zeta+2\pi\mathbf i b^2\sum_{i=1}^3\log S_b\bigg(\frac{\zeta-\mathbf i y_i}{\pi b}\bigg).
\end{split}
\end{equation} 

For $x\in\mathbb C$ with $0<\mathrm{Re}x<\pi,$ let $L(x)$ be the function in \cite[Formula (2.7)]{LMSWY} defined by 
\begin{equation}\label{eq:Lx}
L(x)=x^2-\pi x +\frac{\pi^2}{6}-\mathrm{Li}_2\big(e^{2\mathbf ix}\big),
\end{equation}
where $\mathrm{Li}_2: \mathbb C\setminus (1,\infty)\to\mathbb C$ is the dilogarithm function defined by
$$\mathrm{Li}_2(z)=-\int_0^z\frac{\log (1-u)}{u}du$$
with the integral along any path in $\mathbb C\setminus (1,\infty)$ connecting $0$ and $z.$ In the rest of this paper, for $z\in \mathbb C \setminus (-\infty,0),$ the imaginary part of $\log(z)$ takes value in $(-\pi,\pi).$

Let 
\begin{equation}\label{Vx}
\begin{split}
    V_{\boldsymbol x}(\zeta)=-\frac{\pi^2}{6}+&\sum_{i=1}^3y_i^2-2\sum_{1\leqslant i< j\leqslant 3}y_iy_j-\pi\mathbf i \sum_{i=1}^3y_i + 3\zeta^2+\bigg(\pi-2\mathbf i \sum_{i=1}^3y_i\bigg)\zeta+\sum_{i=1}^3L(\zeta-\mathbf i y_i),
\end{split}
\end{equation}
\begin{equation}\label{kappa}
\begin{split}
    \kappa_{\boldsymbol x}(\zeta)= 2\pi^2 -2\pi \zeta  -\pi \mathbf i \sum_{i=1}^3 \log \Big(1-e^{2\mathbf i (\zeta-\mathbf i y_i) }\Big)
\end{split}
\end{equation}
and 
\begin{equation}\label{nu}
\begin{split}
    \nu_{\boldsymbol x,b}(\zeta)=\frac{V_{\boldsymbol x,b}(\zeta)-V_{\boldsymbol x}(\zeta)-\kappa_{\boldsymbol x}(\zeta)b^2}{b^4}.
\end{split}
\end{equation}
Let 
$$D=\Big\{ \zeta\in\mathbb C\ \Big|\ 0<\mathrm{Re}(\zeta)<\frac{\pi}{6} \Big\}.$$

\begin{proposition}\label{critical2}
    If $\boldsymbol x=(x_1,\dots, x_6)\in\mathbb R^6$  are the edge lengths of a decorated ideal hyperbolic tetrahedron $\Delta,$ then the function $V_{\boldsymbol x}(\zeta)$ has a unique critical point $\zeta^*$ in the domain $D$ with 
    \begin{equation}\label{12}
    \mathrm{Re}\zeta^*\leqslant \frac{\pi}{12}   
    \end{equation}
    and with critical value
    $$V_{\boldsymbol x}(\zeta^*)=-2\mathbf i \mathrm{Cov}(\Delta),$$
    where $\mathrm{Cov}$ is the co-volume of $\Delta.$
\end{proposition}

\begin{proof} We first show that $V_{\boldsymbol x}$ has a unique  critical point in $D.$  For $i\in\{1,2,3\},$ let $u_i=e^{2y_i},$ and let $z=e^{2\mathbf i\zeta}.$ Then by a direct computation, we have
\begin{equation}\label{eq:mod4}
\frac{d V_{\boldsymbol x}}{d \zeta}=2\mathbf i\log\frac{(1-zu_1)(1-zu_2)(1-zu_3)}{-z^3u_1u_2u_3}\quad\quad (\mathrm{mod}\ 4\pi).
\end{equation} 
Here the $\mathrm{mod}\ 4\pi$ ubiquity  comes from the  $2\mathbf i\log.$ Then, as a necessary condition, $\frac{d V_{\boldsymbol x}}{d \zeta}=0$ implies that 
$\frac{(1-zu_1)(1-zu_2)(1-zu_3)}{-z^3u_1u_2u_3}=1,$ which is equivalent to the following quadratic equation 
$$Az^2+Bz+C=0$$
with $A=u_1u_2+u_1u_3+u_2u_3,$ $B=-(u_1+u_2+u_3),$ and $C=1.$
By a direct computation and (\ref{det-}), we have 
\begin{equation}\label{detGram}
    B^2-4AC = 16\det\mathrm{Gram}(\boldsymbol x)<0.
\end{equation}
Let 
\begin{equation}\label{xi}
z^*=e^{2\mathbf i\zeta^*}=\frac{-B+\sqrt{B^2-4AC}}{2A}\quad\text{and}\quad {z^{**}}=e^{2\mathbf i{\zeta^{**}}}=\frac{-B-\sqrt{B^2-4AC}}{2A}
\end{equation}
be the two roots of this quadratic equation. Then by~\eqref{eq:mod4} we have 
\begin{equation}\label{kk'}
    \frac{d V_{\boldsymbol x}}{d \zeta}\Big|_{\zeta=\zeta^*}=4k\pi\quad\text{and}\quad\frac{d V_{\boldsymbol x}}{d \zeta}\Big|_{\zeta={\zeta^{**}}}=4k'\pi
\end{equation}
for some integers $k$ and $k'.$ We claim that, both $\zeta^*$ and $\zeta^{**}$ are  critical points of $V_{\boldsymbol x},$ among the two only $\zeta^*$ is in $D.$   Indeed, at the special case $x_1=\dots=x_6=0,$ ie., $u_1=u_2=u_3=1,$ we have 
$A=3$ and $B=-3.$ Then $z^*=\frac{3+\sqrt{3}\mathbf i}{6}$ and ${z^{**}}=\frac{3-\sqrt{3}\mathbf i}{6}.$ Hence we can choose $\zeta^*=\frac{1}{2\mathbf i}\log \frac{3+\sqrt{3}\mathbf i}{6}\in D;$ and  there is no ${\zeta^{**}}\in D$ making $e^{2\mathbf i{\zeta^{**}}}=\frac{3-\sqrt{3}\mathbf i}{6},$ as $\mathrm{Im}e^{2\mathbf i\zeta}>0$ for all $\zeta\in D.$ A direct computation also shows that at this special case $k=0,$ i.e.,
$$\frac{d V_{\boldsymbol x}}{d\zeta}\Big|_{\zeta=\frac{1}{2\mathbf i}\log \frac{3+\sqrt{3}\mathbf i}{6}}=0.$$ 
Let $\zeta^{**}=\frac{1}{2\mathbf i}\log \frac{3-\sqrt{3}\mathbf i}{6}.$ Then a direct computation shows that  in this special cases $k'=0$ also, i.e.,
$$\frac{d V_{\boldsymbol x}}{d\zeta}\Big|_{\zeta=\frac{1}{2\mathbf i}\log \frac{3-\sqrt{3}\mathbf i}{6}}=0.$$ 
Now for a general $\boldsymbol x,$  we have $A>0,$ $B<0$ and $B^2-4AC<0;$ and by a direct computation, we have by Weitzenb\"ock's inequality that  $\sqrt{4AC-B^2}=\sqrt{-u_1^2-u_2^2-u_3^2+2u_1u_2+2u_1u_3+2u_2u_3}\leqslant \frac{1}{\sqrt{3}}(u_1+u_2+u_3)=\frac{-B}{\sqrt 3}.$ As a consequence, we have $0<\mathrm{Im}z^*\leqslant \frac{\mathrm{Re}z^*}{\sqrt{3}}.$ Hence $\arg z^* \in (0,\frac{\pi}{6}],$ and  we can choose a unique $\zeta^*$ in $D$ with $\mathrm{Re}\zeta^*\leqslant \frac{\pi}{12}.$  For this choice of $\zeta^*,$ we let $\zeta^{**}=-\overline{\zeta^*}$ so that $e^{2\mathbf i\zeta^{**}}=z^{**}$ and $\mathrm{Re}\zeta^{**}\in (-\frac{\pi}{6},0).$
Since $k$ and $k'$ are integer valued continuous functions on the connected space of the edge lengths of decorated ideal hyperbolic tetrahedra, they are constantly $0,$ ie., 
\begin{equation}\label{=0}
\frac{d V_{\boldsymbol x}}{d \zeta}\Big|_{\zeta=\zeta^*}=\frac{d V_{\boldsymbol x}}{d \zeta}\Big|_{\zeta=\zeta^{**}}=0,
\end{equation}
and $\zeta^*$ is the unique critical point of $V_{\boldsymbol x}$ in  $D.$
\bigskip

Next, we compute the value of  $V_{\boldsymbol x}(\zeta^*).$ To this end, for $\boldsymbol x=(x_1,\dots,x_6)\in \mathbb R^6$ and $\zeta\in D,$ we let 
$V(\boldsymbol x,\zeta)\doteq V_{\boldsymbol x}(\zeta),$ and define 
\begin{equation}\label{W}
W(\boldsymbol x)=V({\boldsymbol x},\zeta^*).
\end{equation}

We first verify that 
\begin{equation}\label{co-sch2}
    \frac{\partial W(\boldsymbol x)}{\partial x_k}=-\mathbf i\theta_k
\end{equation}
for each $k\in\{1,\dots,6\};$ and by the symmetry of the indices, it suffices to consider the case that $k=1.$ 
We claim that 
\begin{equation}\label{pW}
\frac{\partial W(\boldsymbol x)}{\partial x_1}=\frac{1}{2}\log\frac{G_{34}-\sqrt{G_{34}^2-G_{33}G_{44}}}{G_{34}+\sqrt{G_{34}^2-G_{33}G_{44}}},
\end{equation}
where $G_{ij}$ is the $ij$-th cofactor of the Gram matrix $\mathrm{Gram}(\Delta).$ Then by (\ref{cos}), we have 
$$\frac{\partial W(\boldsymbol x)}{\partial x_1}=\frac{1}{2}\log e^{-2\mathbf i\theta_1}=-\mathbf i\theta_1.$$
To prove (\ref{pW}), we let $\zeta^{**}=-\overline {\zeta^*}$ so that $z^{**}=e^{2\mathbf i\zeta^{**}},$ and we define 
$$U(\boldsymbol x)= V(\boldsymbol x,\zeta^{**}).$$
We  will prove that 
\begin{enumerate}[(1)]
    \item  
    $$\frac{\partial (W(\boldsymbol x)-U(\boldsymbol x))}{\partial x_1}=\log\frac{G_{34}-\sqrt{G_{34}^2-G_{33}G_{44}}}{G_{34}+\sqrt{G_{34}^2-G_{33}G_{44}}},$$
    and 

    \item   $$\frac{\partial (W(\boldsymbol x)+U(\boldsymbol x))}{\partial x_1}=0,$$
\end{enumerate}
from which (\ref{pW}) follows immediately. To see (1) and (2), considering both $\zeta^*$ and $\zeta^{**}$ as functions of $\boldsymbol x,$ and by the Chain Rule and  (\ref{=0}), we have
$$\frac{\partial W(\boldsymbol x)}{\partial x_1}=\frac{\partial V}{\partial x_1}\bigg|_{\zeta=\zeta^*}+\frac{\partial V}{\partial \zeta}\bigg|_{\zeta=\zeta^*}\cdot\frac{\partial \zeta^*}{\partial x_1}=\frac{\partial V}{\partial x_1}\bigg|_{\zeta=\zeta^*}+\frac{d V_{\boldsymbol x}}{d \zeta}\bigg|_{\zeta=\zeta^*}\cdot\frac{\partial \zeta^*}{\partial x_1}=\frac{\partial V}{\partial x_1}\bigg|_{\zeta=\zeta^*}$$
and
$$\frac{\partial U(\boldsymbol x)}{\partial x_1}=\frac{\partial V}{\partial x_1}\bigg|_{\zeta=\zeta^{**}}+\frac{\partial V}{\partial \zeta}\bigg|_{\zeta=\zeta^{**}}\cdot \frac{\partial \zeta^{**}}{\partial x_1}=\frac{\partial V}{\partial x_1}\bigg|_{\zeta=\zeta^{**}}+\frac{d V_{\boldsymbol x}}{d \zeta}\bigg|_{\zeta=\zeta^{**}}\cdot\frac{\partial \zeta^{**}}{\partial x_1}=\frac{\partial V}{\partial x_1}\bigg|_{\zeta=\zeta^{**}}.$$
Then we have 
\begin{equation*}
\begin{split}
\frac{\partial (W(\boldsymbol x)-U(\boldsymbol x))}{\partial x_1} = & \frac{\partial V}{\partial x_1}\bigg|_{\zeta=\zeta^*}-\frac{\partial V}{\partial x_1}\bigg|_{\zeta=\zeta^{**}}\\ 
=& \log \frac{z^{**}(1-z^*u_1)}{z^*(1-z^{**}u_1)}\quad\quad (\mathrm{mod}\ 2\pi\mathbf i)\\
=& \log \frac{-u_1+u_2+u_3 - \sqrt{u_1^2+u_2^2+u_3^2-2u_1u_2-2u_1u_3-2u_2u_3}}{-u_1+u_2+u_3 +\sqrt{u_1^2+u_2^2+u_3^2-2u_1u_2-2u_1u_3-2u_2u_3}}\\
=&  \log\frac{G_{34}-\sqrt{G_{34}^2-G_{33}G_{44}}}{G_{34}+\sqrt{G_{34}^2-G_{33}G_{44}}} \quad\quad (\mathrm{mod}\ 2\pi\mathbf i),
\end{split}
\end{equation*}
and 
\begin{equation*}
\begin{split}
\frac{\partial (W(\boldsymbol x)+U(\boldsymbol x))}{\partial x_1} = & \frac{\partial V}{\partial x_1}\bigg|_{\zeta=\zeta^*}+\frac{\partial V}{\partial x_1}\bigg|_{\zeta=\zeta^{**}}\\ 
=& \log \frac{(1-z^*u_1)(1-z^{**}u_1)}{z^*z^{**}u_2u_3}\quad\quad (\mathrm{mod}\ 2\pi\mathbf i)\\
=&\, 0 \quad\quad (\mathrm{mod}\ 2\pi\mathbf i),
\end{split}
\end{equation*}
where the last equality comes from that 
$z^*+z^{**}=-\frac{B}{A}=\frac{u_1+u_2+u_3}{u_1u_2+u_1u_3+u_2u_3}$ and $z^*z^{**}=\frac{C}{A}=\frac{1}{u_1u_2+u_1u_3+u_2u_3}.$ To take care of the  $\mathrm{mod}\ 2\pi\mathbf i$ ambiguity, a direct computation at $\boldsymbol x=(0,\dots,0)$ shows that 
$$\frac{\partial (W(\boldsymbol x)-U(\boldsymbol x))}{\partial x_1} =\log\frac{G_{34}-\sqrt{G_{34}^2-G_{33}G_{44}}}{G_{34}+\sqrt{G_{34}^2-G_{33}G_{44}}}=-\frac{2\pi \mathbf i}{3},$$
and 
$$\frac{\partial (W(\boldsymbol x)+U(\boldsymbol x))}{\partial x_1} = 0.$$
This completes the proof of (1) and (2).

Now, by (\ref{co-sch}) and (\ref{co-sch2}), we have $W(\boldsymbol x)=-2\mathbf i \mathrm{Cov}(\Delta)+K$ for some constant $K.$ By a direct computation, 
$$W(0,\dots,0)= -3\mathbf i\mathrm {D}\bigg(\frac{3+\sqrt{3}\mathbf i}{6}\bigg)=-2\mathbf i\mathrm D\bigg(\frac{1+\sqrt{3}\mathbf i}{2}\bigg)=-2\mathbf i\mathrm{Cov}(0,\dots,0),$$
where $\mathrm D$ is the Bloch-Wigner dilogarithm function defined by $\mathrm D(z)=\mathrm{Im}\mathrm{Li}_2(z)+\arg(1-z)\log|z|$ that computes the volume of the ideal hyperbolic tetrahedron of shape parameter $z,$ and the second equality comes from that $\frac{3+\sqrt{3}\mathbf i}{6}$ is the shape parameter of the ideal hyperbolic tetrahedron with dihedral angles $\big(\frac{\pi}{6},\frac{\pi}{6},\frac{2\pi}{3}\big),$ $\frac{1+\sqrt{3}\mathbf i}{2}$ is the shape parameter of the regular ideal hyperbolic tetrahedron, and the two sides differ by a geometric $3$-$2$ Pachner Move. This shows that the constant $K=0,$ and hence 
$$V_{\boldsymbol x}(\zeta^*)=W(\boldsymbol x)=-2\mathbf i \mathrm{Cov}(\Delta).$$   
\end{proof}

\begin{proposition}\label{Hess} Let $\boldsymbol x=(x_1,\dots,x_6)\in\mathbb R^6$  be the edge lengths of a decorated ideal hyperholic tetrahedron, and let $\zeta^*$ be the critical point of $V_{\boldsymbol x}$ in $D.$ 
  Then 
\begin{equation*}
\frac{-V''_{\boldsymbol x}(\zeta^*)}{\exp\big(\frac{\kappa_{\boldsymbol x}(\zeta^*)}{\pi \mathbf i}\big)}=16\sqrt{\det\mathrm{Gram}(\Delta)},
\end{equation*}
where $\mathrm{Gram}(\Delta)$ is the Gram matrix of $\Delta$ in the edge lengths.
\end{proposition}

\begin{proof}
By a direct computation, we have
\begin{equation*}
\begin{split}
\frac{-V''_{\boldsymbol x}(\zeta^*)}{\exp\big({\frac{\kappa_{\boldsymbol x}(\zeta^*)}{\pi \mathbf i}}\big)}=-4\bigg(\frac{1}{1-z^*u_1}+\frac{1}{1-z^*u_2}+\frac{1}{1-z^*u_3}\bigg)\frac{(1-z^*u_1)(1-z^*u_2)(1-z^*u_3)}{z^*}.
\end{split}
\end{equation*}
By a direct computation, this equals
$$-4\bigg(Az^*+2B+\frac{3C}{z^*}\bigg)=4\bigg(Az^*-\frac{C}{z^*}\bigg)$$
with $A=u_1u_2+u_1u_3+u_2u_3,$ $B=-(u_1+u_2+u_3),$ and $C=1,$ which in turn equals
$$4A(z^*-{z^{**}})=4\sqrt{B^2-4AC}=16\sqrt{\det\mathrm{Gram}(\Delta)},$$
where ${z^{**}}$ is the other root of the quadratic equation as in (\ref{xi}). 
\end{proof}

\begin{proposition}\label{concave}
In the domain $D,$ the function $\mathrm{Im}V_{\boldsymbol x}(\zeta)$ is strictly concave down in $\mathrm{Im}\zeta.$
\end{proposition}

\begin{proof}  By a direct computation, we have
\begin{equation}\label{2nd}
\frac{\partial ^2 \mathrm{Im}V_{\boldsymbol x}(\zeta)}{\partial \mathrm{Im}\zeta^2}
=-\sum_{i=1}^3\frac{\sin\big(2\mathrm{Re}(\zeta-\mathbf iy_i)\big)}{\cosh^2\mathrm{Im}(\zeta-\mathbf iy_i)-\cos^2\mathrm{Re}(\zeta-\mathbf iy_i)}.
\end{equation}
For  $\zeta\in D,$ we have $\mathrm{Re}(\zeta-\mathbf iy_i) \in (0,\frac{\pi}{6})$ for $i\in\{1,2,3\},$ hence $\sin\big(2\mathrm{Re}(\zeta-\mathbf iy_i)\big)>0.$ Then by (\ref{2nd}),  we have
\begin{equation*}
\frac{\partial ^2 \mathrm{Im}V_{\boldsymbol x}(\zeta)}{\partial \mathrm{Im}\zeta^2}<0,\\
\end{equation*}
and $\mathrm{Im}V_{\boldsymbol x}$ is strictly concave down in $\mathrm{Im}\zeta.$
\end{proof}

\begin{proposition} \label{limder}
Let $c\in(0,\frac{\pi}{6}).$ Then on $\Gamma_c=\Big\{\zeta\in D\ \Big|\ \mathrm{Re}\zeta =c  \Big\},$ we have 
$$\lim_{\mathrm{Im}\zeta\to +\infty}\frac{\partial \mathrm{Im}V_{\boldsymbol x}(\zeta)}{\partial \mathrm{Im}\zeta} = 12c-2\pi \quad\text{and}\quad \lim_{\mathrm{Im}\zeta\to -\infty}\frac{\partial \mathrm{Im}V_{\boldsymbol x}(\zeta)}{\partial \mathrm{Im}\zeta} = 4\pi.$$
\end{proposition}

\begin{proof} By a direct computation, we have
$$ \frac{\partial \mathrm{Im}L( x +\mathbf il)}{\partial l} =2 x -\pi-2\arg\Big(1-e^{-2l+2\mathbf i x}\Big).$$
As a consequence, for a fixed $ x\in (0,\frac{\pi}{2}),$ we have
$$\lim_{l\to +\infty} \frac{\partial \mathrm{Im}L( x +\mathbf il)}{\partial l} =2 x -\pi -2\arg(1)=2 x -\pi$$
and
$$\lim_{l\to -\infty} \frac{\partial \mathrm{Im}L( x +\mathbf il)}{\partial l} =2 x -\pi-2\arg(-e^{2\mathbf i x })=\pi-2 x .$$
Then
$$\frac{\partial \mathrm{Im}V_{\boldsymbol x}(\zeta)}{\partial \mathrm{Im}\zeta} =6\mathrm{Re}\zeta+\pi+\sum_{i=1}^3\frac{\partial \mathrm{Im}L(\zeta-\mathbf iy_i)}{\partial \mathrm{Im}\zeta};$$
and as a consequence, we have on $\Gamma _c$ that 
$$\lim_{\mathrm{Im}\zeta\to +\infty}\frac{\partial \mathrm{Im}V_{\boldsymbol x}(\zeta)}{\partial \mathrm{Im}\zeta}=6\mathrm{Re}\zeta+\pi+\sum_{i=1}^3\big(2\mathrm{Re}(\zeta-\mathbf iy_i)-\pi\big)=12c-2\pi,$$
and
$$\lim_{\mathrm{Im}\zeta\to -\infty}\frac{\partial \mathrm{Im}V_{\boldsymbol x}(\zeta)}{\partial \mathrm{Im}\zeta}=6\mathrm{Re}\zeta+\pi+\sum_{i=1}^3\big(\pi-2\mathrm{Re}(\zeta-\mathbf iy_i)\big)=4\pi.$$
\end{proof}

 \begin{proposition}\label{bound} 
 For $\delta >0$, let  $$D_{\delta}=\bigg\{\zeta\in \mathbb C\ \bigg|\ \mathrm {Re}\zeta\in\Big[\delta,\frac{\pi}{6}-\delta\Big]\bigg\}.$$
 
 \begin{enumerate}[(1)]
 \item 
 For $\delta> 0$  sufficiently small, there exists a constant $K=K_{\delta}>0$ such that 
$$\Bigg|\frac{\partial \mathrm{Im}\kappa_{\boldsymbol x}(\zeta)}{\partial\mathrm{Im}\zeta}\Bigg|<K$$
for all $\boldsymbol x\in\mathbb R^6$ and all $\zeta\in D_{\delta}.$
\item  
For $b>0$ and $\delta> 0$ both sufficiently small, there exists a constant $N=N_{\delta}>0$ independent of $b$ and of  $\boldsymbol x\in\mathbb R^6$ such that 
$$\mathrm{Im}\nu_{\boldsymbol x,b}(\zeta) \leqslant \big|\nu_{\boldsymbol x,b}(\zeta)\big|<N$$
for all $\zeta\in D_{\delta}.$
\end{enumerate}
\end{proposition}

The proof of Proposition \ref{bound} follows verbatim that of \cite[Proposition 3.10 and Proposition 3.11]{LMSWY}.

\subsection{Proof of Theorem \ref{6jasymp}} \label{3.3}

\begin{proof}[Proof of Theorem \ref{6jasymp}]
Recall the domain $D=\Big\{ \zeta\in\mathbb C\ \Big|\ 0< \mathrm{Re}\zeta < \frac{\pi}{6} \Big\}$ and let
$$\Gamma^*=\{\zeta\in D\ |\ \mathrm{Re}\zeta=\mathrm{Re}\zeta^*\}.$$ 
Let  $d>0$ be sufficiently small so that the region 
 $$B_{d}=\Big\{\zeta\in \mathbb C \ \Big|\ |\mathrm{Re}\zeta - \mathrm{Re}\zeta^*| \leqslant d \text{ and } |\mathrm{Im}\zeta-\mathrm{Im}\zeta^*|\leqslant d \Big\}$$
 lies entirely in $D.$ By Proposition \ref{concave}, $\frac{\partial \mathrm{Im}V_{\boldsymbol x}(\zeta)}{\partial \mathrm{Im}\zeta}$ is strictly decreasing as $\zeta$ moves along $\Gamma^*$ from $\zeta^*$ with $\mathrm{Im}\zeta$ approaching $+\infty,$ and is strictly increasing as $\zeta$ moves along $\Gamma^*$ from $\zeta^*$ with $\mathrm{Im}\zeta$ approaching $-\infty.$ 
Therefore, there is an $\epsilon>0$ sufficiently small so that 
 \begin{equation}\label{ImU}
\mathrm{Im} V_{\boldsymbol x}(\zeta^*\pm \mathbf il )< -2\mathrm{Cov}(\Delta)-4\epsilon
\end{equation}
and for $l>d;$ and by Proposition \ref{concave},  Proposition \ref{limder} and (\ref{12}),  there is an $L>d$ such that 
\begin{equation}\label{Dv1}
 \frac{\partial \mathrm{Im} V_{\boldsymbol x}}{\partial\mathrm{Im}\zeta } (\zeta^*+  \mathbf i l)<-\pi + \epsilon\quad\text{and}\quad  \frac{\partial \mathrm{Im} V_{\boldsymbol x}}{\partial\mathrm{Im}\zeta} (\zeta^* -  \mathbf i l)>\pi-\epsilon
\end{equation}
for $l>L.$  
Let
$$\Gamma^*_d=\Gamma^*\cap B_d=\Big\{ \zeta \in \Gamma^*\ \Big|\ |\mathrm{Im}\zeta-\mathrm{Im}\zeta^*|<d\Big\}$$
and let 
$$\Gamma^*_L=\Big\{ \zeta \in \Gamma^*\ \Big|\ |\mathrm{Im}\zeta-\mathrm{Im}\zeta^*|\leqslant L\Big\}.$$
We will show that, as $b\to 0,$
\begin{enumerate}[(I)]
\item 
$$\frac{1}{\pi b}\int_{\Gamma^*_d}\exp\bigg(\frac{V_{\boldsymbol  x}(\zeta)+ \kappa_{\boldsymbol  x}(\zeta)b^2+ \nu_{\boldsymbol  x,b}(\zeta)b^4}{2\pi \mathbf i b^2} \bigg)d\zeta = \frac{1}{2} \frac{e^{\frac{-\mathrm{Cov}(\Delta)}{\pi b^2}}}{\sqrt[4]{-\det\mathrm{Gram}(\Delta)}}\Big(1+O\big(b^2\big)\Big),
$$

\item $$\bigg|\frac{1}{\pi b}\int_{\Gamma^*_L\setminus \Gamma^*_d}\exp\bigg(\frac{V_{\boldsymbol  x}(\zeta)+ \kappa_{\boldsymbol  x}(\zeta)b^2+ \nu_{\boldsymbol  x,b}(\zeta)b^4}{2\pi \mathbf i b^2} \bigg)d\zeta\bigg|< O\Big(e^{\frac{-\mathrm{Cov}(\Delta)-\epsilon_1}{\pi b^2}}\Big),$$
and 

\item $$\bigg|\frac{1}{\pi b}\int_{\Gamma^*\setminus \Gamma^*_L}\exp\bigg(\frac{V_{\boldsymbol  x}(\zeta)+ \kappa_{\boldsymbol  x}(\zeta)b^2+ \nu_{\boldsymbol  x,b}(\zeta)b^4}{2\pi \mathbf i b^2} \bigg)d\zeta\bigg|<O\Big(e^{\frac{-\mathrm{Cov}(\Delta)-\epsilon_1}{\pi b^2}}\Big)$$
\end{enumerate}
for some $\epsilon_1>0,$ from which the result follows.
\medskip

For (I), we claim that  all the conditions of Proposition \ref{saddle} are satisfied by letting $\hbar=b^2,$ $D=B_d,$ $f=\frac{V_{\boldsymbol  x}}{2\pi \mathbf i},$ $g=\exp\big(\frac{\kappa_{\boldsymbol  x}}{2\pi \mathbf i }\big),$ $f_\hbar=\frac{V_{\boldsymbol  x}+\nu_{\boldsymbol  x,b}b^4}{2\pi \mathbf i},$ $\upsilon_\hbar=\frac{\nu_{\boldsymbol  x,b}}{2\pi \mathbf i},$ $S=\Gamma^*_d$ and $c=\zeta^*.$ 

Indeed, by Proposition \ref{critical2}, $\zeta^*$ is a critical point of $f=\frac{V_{\boldsymbol  x}}{2\pi \mathbf i}$ in $B_{d},$ hence condition (i) is satisfied. 

By Propositions \ref{critical2} and \ref{concave}, $\zeta^*$ is the unique maximum point of $\mathrm{Re}f=\frac{\mathrm{Im}V_{\boldsymbol x}}{2\pi}$ on $\Gamma^*_d,$ hence condition (ii) is satisfied.

For conditions (iii) and (iv), since $\kappa_{\boldsymbol  x}(\zeta^*)$ is a finite value,  $g(\zeta^*)=\exp\big(\frac{\kappa_{\boldsymbol  x}(\zeta^*)}{2\pi \mathbf i }\big)\neq 0,$ and condition (iv) is satisfied. Also by this and Proposition \ref{Hess}, $V''_{\boldsymbol  x}(\zeta^*)=-16{\exp\big(\frac{\kappa_{\boldsymbol  x}(\zeta^*)}{\pi \mathbf i}\big)}\sqrt{\det\mathrm{Gram}(\Delta)}\neq 0,$ and condition (iii) is satisfied.

For condition (v), by Proposition \ref{bound},  $|\upsilon_{\hbar}(\zeta)|=\big|\frac{\nu_{\boldsymbol  x,b}(\zeta)}{2\pi \mathbf i}\big|<\frac{N}{2\pi}$ on $B_{d}.$ 

For condition (vi), since $\Gamma^*$ is a straight line, it is a smooth embedding near $\zeta^*.$

Finally, by Proposition \ref{saddle}, Proposition \ref{critical2}  and Proposition \ref{Hess}, we have as $b\to 0,$
\begin{equation*}
\begin{split}
\frac{1}{\pi b}\int_{\Gamma^*_d}\exp\bigg(\frac{V_{\boldsymbol  x,b}(\zeta)}{2\pi \mathbf ib^2}\bigg) d\zeta=& \frac{(2\pi b^2)^\frac{1}{2}}{\pi b} \frac{\exp\big(\frac{\kappa_{\boldsymbol  x}(\zeta^*)}{2\pi \mathbf i}\big)}{\sqrt{-\frac{V_{\boldsymbol  x}''(\zeta^*)}{2\pi \mathbf i}}}e^{\frac{V_{\boldsymbol x}(\zeta^*)}{2\pi \mathbf i b^2}}\Big(1+O\big(b^2\big)\Big)\\
=&\frac{1}{2}\frac{e^{\frac{-\mathrm{Cov}(\Delta)}{\pi b^2}}}{\sqrt[4]{-\det\mathrm{Gram}(\Delta)}}\Big(1+O\big(b^2\big)\Big).
\end{split}
\end{equation*}
This completes the proof of (I). 
\medskip

For (II) and (III), we have
\begin{equation*}
\begin{split}
& \bigg|\frac{1}{\pi b}\int_{\Gamma^*\setminus \Gamma^*_d}\exp\bigg(\frac{V_{\boldsymbol  x}(\zeta)+ \kappa_{\boldsymbol  x}(\zeta)b^2+ \nu_{\boldsymbol  x,b}(\zeta)b^4}{2\pi \mathbf i b^2} \bigg)d\zeta\bigg|\\
\leqslant & \frac{1}{\pi b}\int_{\Gamma^*\setminus \Gamma^*_d}\exp\bigg(\frac{\mathrm{Im}V_{\boldsymbol  x}(\zeta)+\mathrm{Im}\kappa_{\boldsymbol  x}(\zeta)b^2+\mathrm{Im}\nu_{\boldsymbol  x,b}(\zeta)b^4}{2\pi b^2} \bigg)|d\zeta|.
\end{split}
\end{equation*}
\smallskip

For  (II), by Proposition \ref{bound}, there is a  $b_1>0$ such that 
\begin{equation}\label{last2}
\mathrm{Im}\nu_{\boldsymbol x,b}(\zeta)b^4<Nb^4<\epsilon 
\end{equation}
for all $b<b_1$ and for all $\zeta\in\Gamma^*;$  and together with  (\ref{ImU}), we have
\begin{equation}\label{last3}
\mathrm{Im}V_{\boldsymbol  x}(\zeta)+\mathrm{Im}  \nu_{\boldsymbol  x,b}(\zeta)b^4< -2\mathrm{Cov}(\Delta) -3\epsilon
\end{equation}
for all $b<b_1$ and  $\zeta\in\Gamma^*\setminus \Gamma^*_d.$ By the compactness of ${\Gamma^*_L}\setminus \Gamma^*_d,$ there exists an $M>0$ such that 
\begin{equation}\label{Imk}
\mathrm{Im}\kappa_{\boldsymbol x}(\zeta)<M
\end{equation}
for all $\zeta\in\Gamma^*_L\setminus \Gamma^*_d.$ As a consequence of (\ref{last3}) and (\ref{Imk}), we have
\begin{equation*}
\begin{split}
&\frac{1}{\pi b}\int_{\Gamma^*_L\setminus \Gamma^*_d}\exp\bigg(\frac{\mathrm{Im}V_{\boldsymbol  x}(\zeta)+\mathrm{Im}\kappa_{\boldsymbol  x}(\zeta)b^2+\mathrm{Im}\nu_{\boldsymbol  x,b}(\zeta)b^4}{2\pi b^2} \bigg)|d\zeta|\\
<  & \frac{2(L-d) e^{\frac{M}{2\pi}}}{\pi b} \exp\bigg(\frac{-\mathrm{Cov}(\Delta)-\epsilon }{\pi b^2}   \bigg )< O\Big(e^{\frac{-\mathrm{Cov}(\Delta)-\epsilon_1}{\pi b^2}}\Big)
\end{split}
\end{equation*}
for any $\epsilon_1<\epsilon.$ This completes the proof of (II). 
\smallskip

For (III), there  is a $b_0\in (0, b_1)$ such that for all $b<b_0,$
\begin{equation}\label{ImK}
\mathrm{Im}\kappa_{\boldsymbol x}(\zeta^*\pm \mathbf i L) b^2<\epsilon,
\end{equation}
$Kb^2<\epsilon$ and $Nb^4<\epsilon,$ where $K$ and $N$ are respectively the constants in Proposition \ref{bound} (1) and (2). We claim that, for $\zeta\in\Gamma^*\setminus \Gamma^*_L$ and $b<b_0,$ 
\begin{equation}\label{cl}
\mathrm{Im}V_{\boldsymbol  x}(\zeta)+\mathrm{Im}\kappa_{\boldsymbol  x}(\zeta)b^2+\mathrm{Im}\nu_{\boldsymbol  x,b}(\zeta)b^4<-2\big(\mathrm{Cov}(\Delta)+\epsilon\big)-(\pi-2\epsilon)\big(|\zeta-\zeta^*|-L\big),
\end{equation}
as a consequence of which, we have,
\begin{equation}\label{CI}
\begin{split}
& \frac{1}{\pi b}\int_{\Gamma^*\setminus \Gamma^*_L}\exp\bigg(\frac{\mathrm{Im}V_{\boldsymbol  x}(\zeta)+\mathrm{Im}\kappa_{\boldsymbol  x}(\zeta)b^2+\mathrm{Im}\nu_{\boldsymbol  x,b}(\zeta)b^4}{2\pi b^2} \bigg)|d\zeta| \\
 < & \frac{1}{\pi b} \exp\bigg(\frac{-\mathrm{Cov}(\Delta)-\epsilon}{\pi b^2}\bigg)\int_{\Gamma^*\setminus \Gamma^*_L}\exp\bigg(\frac{-(\pi-2\epsilon)\big(|\zeta-\zeta^*|-L\big)}{2\pi}\bigg) |d\zeta|\\
< & O\Big(e^{\frac{-\mathrm{Cov}(\Delta)-\epsilon_1}{\pi b^2}}\Big)
\end{split}
\end{equation}
for any $\epsilon_1<\epsilon.$ 

For the proof of the claim, by (\ref{Dv1}) and the choice of $b_0,$ for $l>L,$ we have 
$$\frac{\partial}{\partial \mathrm{Im}\zeta} \Big(\mathrm{Im}V_{\boldsymbol x}(\zeta^*+\mathbf il)+\mathrm{Im}\kappa_{\boldsymbol x}(\zeta^*+\mathbf il)b^2\Big)<-\pi+2\epsilon,$$
and 
$$\frac{\partial}{\partial \mathrm{Im}\zeta} \Big(\mathrm{Im}V_{\boldsymbol x}(\zeta^*-\mathbf il)+\mathrm{Im}\kappa_{\boldsymbol x}(\zeta^*-\mathbf il)b^2\Big)>\pi-2\epsilon.$$
Together with the Mean Value Theorem, (\ref{ImU}) and (\ref{ImK}), we have 
\begin{equation}\label{Bou}
\begin{split}
&\mathrm{Im}V_{\boldsymbol x}(\zeta)+ \mathrm{Im}\kappa_{\boldsymbol x}(\zeta)b^2 \\
< & \mathrm{Im}V_{\boldsymbol x}(\zeta^*\pm \mathbf iL) + \mathrm{Im}
\kappa_{\boldsymbol x}(\zeta^*\pm \mathbf iL)b^2 - (\pi-2\epsilon) \big |  \zeta - (\zeta^*\pm \mathbf iL ) \big|\\
< & -2\mathrm{Cov}(\Delta)-3\epsilon  -(\pi-2\epsilon)\big(|\zeta-\zeta^*|-L \big)
\end{split}
\end{equation}
for all $\zeta \in \Gamma^*\setminus \Gamma^*_L.$  Finally, putting (\ref{Bou}) and (\ref{last2}) together, we have  (\ref{cl}) and the first  inequality in (\ref{CI}); and since 
$$|\zeta-\zeta^*|-L\to+\infty$$
as $\zeta \in \Gamma^*\setminus \Gamma^*_L$ approaches $\infty,$ we have the second inequality in (\ref{CI}). This completes the proof of (III).
\medskip

Putting (I), (II), (III) and (\ref{idealb6j}) together, we have as $b\to 0,$ 
$$\bigg|\begin{matrix} a_1 & a_2 & a_3 \\ a_4 & a_5 & a_6 \end{matrix} \bigg|_b =\frac{1}{2}\frac{e^{\frac{-\mathrm{Cov}(\Delta)}{\pi b^2}}}{\sqrt[4]{-\det\mathrm{Gram}(\Delta)}}\Big(1+O\big(b^2\big)\Big). $$ 
\end{proof}


\section{$\mathrm {U}_{q\tilde q}\mathfrak{sl}(2;\mathbb R)$ Turaev-Viro invariants for cusped  $3$-manifolds}

The goal of this section is to prove Theorem \ref{converge} and Theorem \ref{topinv}.

Let $V$ be the set of boundary components of $M.$ Recall that the $|E|\times |V|$-matrix $A$ is the adjacency matrix of the edges and the boundary components given by $A_{e,v}=|e\cap v|,$ the number of points of intersections of the edge $e$ with the boundary component $v.$  The $\mathbb R^V$-action on $\big(\frac{Q}{2}+\mathbf i \mathbb R\big)^E$ is given by $\boldsymbol u\cdot \boldsymbol a = \boldsymbol a + \frac{\mathbf i}{2\pi b} A \boldsymbol u;$ and the  quotient measure $d
\mu$ on $\big(\frac{Q}{2}+\mathbf i \mathbb R\big)^E/\mathbb R^V$ is obtained by disintegrating the Lebesgue measure $\prod_{e\in E} d\mathrm{Im}a_e$ on $\big(\frac{Q}{2}+\mathbf i\mathbb R\big)^ E$ along the orbits $[\boldsymbol a]= \boldsymbol a + \frac{\mathbf i}{2\pi b} A(\mathbb R^V)$ of the action. 

In the concrete computations, 
we  consider $A$ as a linear map from $\mathbb R^V$ to $\big(\frac{Q}{2}+\mathbf i\mathbb R\big)^E,$ where the vector space structure of 
$\big(\frac{Q}{2}+\mathbf i\mathbb R\big)^E$ is induced from that of $\mathbb R^E$ via the bijection between these two spaces defined by sending $\boldsymbol a$ to $\mathrm{Im}\boldsymbol a.$ 
It is easy to see that $A$ is injective, 
hence is of full rank and $\det(A^TA)\neq 0,$ where $A^T$ is the transpose of $A.$
Let $X$ be the orthogonal complement of $A(\mathbb R^V)$ in $\big(\frac{Q}{2}+\mathbf i\mathbb R\big)^E,$ that is,
$$X=\Bigg\{ \boldsymbol a\in  \bigg(\frac{Q}{2}+i\mathbb R\bigg)^E \ \Bigg|\  \mathrm{Im}\boldsymbol a\cdot A(\boldsymbol u) = 0\ \text{for all }\boldsymbol u\in \mathbb R^V\Bigg\}.$$ Then the integral (\ref{b-TV}) can be computed  as 
\begin{equation}\label{b-TV2}
   \mathrm{TV}_b(M,\mathcal T)= \Big(\frac{1}{2\pi b}\Big)^{|V|} \int _ X \prod_{e\in E} |e|_{\boldsymbol{a}}\prod_{\Delta\in T}|\Delta|_{\boldsymbol a}\sqrt{\det(A^T A)}dm_X(\boldsymbol a),
\end{equation}
where $dm_X$ is the Lebesgue measure on $X$ obtained by restricting the Lebesgue measure $\prod_{e\in E}d\mathrm{Im}a_e$ on $\big(\frac{Q}{2}+\mathbf i\mathbb R\big)^E.$

\subsection{Proof of Theorem \ref{converge}}

 For the proof of Theorem \ref{converge} and Theorem \ref{vol}, we need the following

\begin{lemma}\label{inj} 
For each $\Delta\in \mathcal T, $ we denote its edges by $e_{\Delta,1},\dots,e_{\Delta,6},$ with  $e_j$ and $e_{j+3}$ being the opposite to each other for $j\in\{1,2,3\}$.  For $\boldsymbol a \in \big(\frac{Q}{2}+\mathbf i \mathbb R\big)^E,$ we let 
$$\boldsymbol b_\Delta(\boldsymbol a)=(b_{\Delta,1}(\boldsymbol a),b_{\Delta,2}(\boldsymbol a),b_{\Delta,3}(\boldsymbol a)),$$
where  for each $j\in \{1,2,3\},$
$$b_{\Delta,j}(\boldsymbol a)=\mathrm{Im}a_{\Delta,j}+\mathrm{Im}a_{\Delta,j+3}-\frac{1}{3}\sum_{k=1}^6\mathrm{Im}a_{\Delta,k}.$$ 
Let $B: \big(\frac{Q}{2}+\mathbf i\mathbb R\big)^E \to (\mathbb R^3)^T$ be the linear map defined by
$$B(\boldsymbol a)= \big(\boldsymbol b_{\Delta}(\boldsymbol a)\big)_{\Delta\in T}.$$
Then the restriction 
$B|_X: X \to (\mathbb R^3)^T$
of $B$ to $X$ is injective. 
\end{lemma}

\begin{proof}
It suffices to prove that 
$\ker B = A(\mathbb R^V).$
\medskip

We first prove that $\ker B\subset A(\mathbb R^V).$ Let $\boldsymbol a\in \ker B.$ Then $B(\boldsymbol a)=\big(\boldsymbol b_\Delta(\boldsymbol a)\big)_{\Delta\in T}=\boldsymbol 0$ implies that 
for each $\Delta\in T$ with the edges $e_{\Delta,1},\dots,e_{\Delta,6}$ such that $e_{\Delta,j}$ and $e_{\Delta,j+3}$ are opposite for each $j\in\{1,2,3\},$
\begin{equation}\label{==}
\mathrm{Im}a_{\Delta,1}+ \mathrm{Im}a_{\Delta,4}= \mathrm{Im}a_{\Delta,2}+ \mathrm{Im}a_{\Delta,5}= \mathrm{Im}a_{\Delta,3}+ \mathrm{Im}a_{\Delta,6}.   
\end{equation}

We define a vector $\boldsymbol u=(u_v)_{v\in V}\in\mathbb R^V$ as follows: For each boundary component $v\in V$ of $M,$ let $F$ be a face of $\Delta$ that intersects $v,$ let $e_i$ and $e_j$ be two edges of $F$ intersecting $v$ and let $e_k$ be the other edge of $F.$ Then we let 
$$u_v=\frac{\mathrm{Im}a_i+\mathrm{Im}a_j-\mathrm{Im}a_k}{2}.$$
We claim that $u_v$ is independent of the choice of the triple $(F,e_i,e_j),$ hence is a  well-defined vector in $\mathbb R^V.$  Indeed, any two such triples $(F,e_i,e_j)$ and $(F',e_i',e_j')$ can be related by a sequence of triples $\{(F_k,e_{i_k},e_{j_k})\}_{k\in\{1,\dots,n\}},$ such that for each two adjacent triples $(F_k,e_{i_k},e_{j_k})$ and $(F_{k+1},e_{i_{k+1}},e_{j_{k+1}}),$ $F_k$ and $F_{k+1}$ are the faces of a common tetrahedron $\Delta$ and $\{e_{i_k},e_{j_k}\}\cap \{e_{i_{k+1}},e_{j_{k+1}}\}\neq \emptyset.$ Then applying (\ref{==}) to $\Delta,$ we have that the two adjacent triples give the same value of $u_v,$ which implies that the two original triples give the same value of $u_v.$ This proves the claim.

Next, we show that $A(\boldsymbol u)=\boldsymbol a.$ For each $e\in E,$ let $v_1, v_2\in V$ that intersect $e,$ and let $F$ be a face containing $e.$ Let $e_1$ and $e_2$ be the other two edges of $F.$ 
Then
\begin{equation*}
\begin{split}
A(\boldsymbol u)_e= & \frac{Q}{2}+ \mathbf i (u_{v_1}+u_{v_2}) 
 = \frac{Q}{2}+ \mathbf i \bigg( \frac{\mathrm{Im}a_e+\mathrm{Im}a_{e_1}-\mathrm{Im}a_{e_2}}{2}+  \frac{\mathrm{Im}a_e+\mathrm{Im}a_{e_2}-\mathrm{Im}a_{e_1}}{2} \bigg) = a_e. 
\end{split}
\end{equation*}
This implies that $A(\boldsymbol u)=\boldsymbol a,$
completing the proof that  $\ker B\subset A(\mathbb R^V).$ 
\medskip

Next we prove that $A(\mathbb R^V)\subset \ker B.$ Let $\boldsymbol u\in \mathbb R^V,$ and let $\boldsymbol a=A(\boldsymbol u)\in \big(\frac{Q}{2}+\mathbf i\mathbb R\big)^E.$ Then for each $e\in E$ with $v_1,v_2\in V$ intersection $e,$ we have 
$$\mathrm{Im}a_e= u_{v_1}+u_{v_2}.$$
Now for each $\Delta\in T$ with  $v_1,v_2,v_3, v_4\in V$ intersection $\Delta,$  we have 
$$\mathrm{Im}a_1+\mathrm{Im}a_4=\mathrm{Im}a_2+\mathrm{Im}a_5=\mathrm{Im}a_3+\mathrm{Im}a_6=\sum_{i=1}^4u_{v_i}.$$
As a consequence, $b_{\Delta,1}(\boldsymbol a)=b_{\Delta,2}(\boldsymbol a)=b_{\Delta,3}(\boldsymbol a)=0.$ This implies that 
$B(A(\boldsymbol u))=B(\boldsymbol a)=\boldsymbol 0,$ and hence $A(\mathbb R^V)\subset \ker B.$ 
\end{proof}

\begin{proof}[Proof of Theorem \ref{converge}]
Let $\boldsymbol \theta$ be any angle structure of $(M,\mathcal T).$ For each $\Delta\in T$ with edges $(e_1,\dots,e_6),$ we let $\boldsymbol \theta_\Delta=(\theta_{(\Delta,e_1)},\dots,\theta_{(\Delta,e_6)})$ be the six dihedral angles of $\Delta$ assigned by $\boldsymbol \theta;$ and for each $\boldsymbol a \in \big(\frac{Q}{2}+\mathbf i\mathbb R\big)^E,$ we let $\mathrm{Im}\boldsymbol a_\Delta =(\mathrm{Im}a_{e_1},\dots, \mathrm{Im}a_{e_6}).$ Then we have 
\begin{equation}\label{anglesum}
2\pi \sum_{e\in E} \mathrm{Im}a_e=\sum_{e\in E}\bigg(\sum_{\Delta\sim e}\theta_{(\Delta,e)}\bigg)\mathrm{Im}a_e  = \sum_{\Delta\in T} \boldsymbol \theta_\Delta\cdot  \mathrm{Im} \boldsymbol a_\Delta.
\end{equation}
Now by (\ref{b-TV2}), we have 
\begin{equation*}
\begin{split}
    \big|\mathrm{TV}_b( M,\mathcal T)\big| &
\leqslant \Big(\frac{1}{2\pi b}\Big)^{|V|}\sqrt{\det(A^TA)}\int_X \bigg| e^{2\pi Q\sum_{e\in E}\mathrm{Im}a_e} \prod_{\Delta\in T}|\Delta|_{\boldsymbol a}\bigg|dm_X(\boldsymbol a)\\
    &=\Big(\frac{1}{2\pi b}\Big)^{|V|}\sqrt{\det(A^TA)}\int_X \prod_{\Delta\in T}\bigg| e^{ Q(\boldsymbol \theta_\Delta\cdot \mathrm{Im}\boldsymbol a_\Delta)} |\Delta|_{\boldsymbol a}\bigg|dm_X(\boldsymbol a)\\
    &<\Big(\frac{1}{2\pi b}\Big)^{|V|}\sqrt{\det(A^TA)} \prod_{\Delta\in T} C_{b,\boldsymbol \theta_\Delta}\int_X  e^{-\sum_{\Delta\in T}\epsilon_{\boldsymbol \theta_\Delta}\max\big\{b_{\Delta,1}(\boldsymbol a),b_{\Delta,2}(\boldsymbol a),b_{\Delta,3}(\boldsymbol a)\big\}}dm_X(\boldsymbol a)\\
    &< +\infty,
\end{split}
\end{equation*}
where  the equality comes from (\ref{anglesum}), the second inequality comes from Corollary \ref{thetabj} with $C_{b,\boldsymbol \theta_\Delta}$ and $\epsilon_{\boldsymbol \theta_\Delta}$ the constants therein, and the last inequality comes from Lemma \ref{inj} that $B$ embeds $X$ into  $F^T,$ where  $F=\big\{(b_1,b_2,b_3)\in\mathbb R^3\ \big|\ b_1+b_2+b_3=0\big\},$ and that $\max\{b_1,b_2,b_3\}\geqslant \sqrt{\frac{b_1^2+b_2^2+b_3^2}{6}}$ on $F,$ hence
the function  $e^{-\sum_{\Delta\in T}\epsilon_{\boldsymbol \theta_\Delta}\max\{x_{\Delta,1},x_{\Delta,2},x_{\Delta,3}\}}$ is integrable on $F^T.$
\end{proof}

\subsection{Proof of Theorem \ref{topinv}}

For the proof of Theorem \ref{topinv}, we need the following 
\begin{lemma}\label{P}

Let $U$ be any subspace of  $\big(\frac{Q}{2}+\mathbf i\mathbb R\big)^E$ such that
$$U + A(\mathbb R^V) = \bigg(\frac{Q}{2}+\mathbf i\mathbb R\bigg)^E\quad\text{and}\quad U\cap A(\mathbb R^V)=\{\boldsymbol 0\}.$$
 Let $P$ be the $|E|\times (|E|-|V|)$-matrix consisting of any basis of $U$ as the colums, and let $[A,P]$ be the $|E|\times |E|$ matrix obtained by juxtaposing  $A$ and $P.$ Then 
 $$
  \mathrm{TV_b}(M,\mathcal T)= \Big(\frac{1}{2\pi b}\Big)^{|V|} \int _ U \prod_{e\in E} |e|_{\boldsymbol{a}}\prod_{\Delta\in T}|\Delta|_{\boldsymbol a}\frac{\big|\det[A,P]\big|}{\sqrt{\det(P^T P)}}dm_U(\boldsymbol a),
 $$
where $dm_U$ is the Lebesgue measure on $U$ obtained by restricting the Lebesgue measure $\prod_{e\in E}d\mathrm{Im}a_e$ on $\big(\frac{Q}{2}+i\mathbb R\big)^E.$ 
\end{lemma}

\begin{proof} 
We first observe that the orthogonal complement $X$ of $A(\mathbb R^V)$ is a special case of such $U.$ Let $O$ be the $|E|\times(|E|-|V|)$-matrix consisting of an orthonormal basis of $X$ as the columns, and let $[A,O]$ be the $|E|\times |E|$-matrix obtained by juxtaposing $A$ and $O.$ Then 
\begin{equation}\label{AO}
\big|\det[A,O]\big|=\sqrt{\det (A^TA)}.   
\end{equation}

Let $\{\boldsymbol p_i\}_{i\in \{1,\dots, |E|-|V|\}}$ be the columns of $P,$ and let $\{\boldsymbol o_i\}_{i\in \{1,\dots, |E|-|V|\}}$ be the columns of $O.$ 

We next consider the spacial case that for each $i\in \{1,\dots, |E|-|V|\},$ $\boldsymbol p_i=\boldsymbol o_i + \boldsymbol v_i$
for some vector $\boldsymbol v_i \in A(\mathbb R^V).$ In this case, as $P$ is obtained from $O$ by adding suitable linear combinations of columns of $A,$ we have 
\begin{equation}\label{OP}
\big|\det[A,P]\big|=\big|\det[A,O]\big|.
\end{equation}
For $\boldsymbol a\in \big(\frac{Q}{2}+\mathbf i\mathbb R\big)^E,$ let 
$$F(\boldsymbol a)= \bigg(\frac{1}{2\pi b}\bigg)^{|V|} \prod_{e\in E}|e|_{\boldsymbol a}\prod_{\Delta\in T} |\Delta|_{\boldsymbol a}.$$
Then for each $\boldsymbol x\in \mathbb R^{|E|-|V|},$ we have 
$P\cdot \boldsymbol x - O \cdot \boldsymbol x =[\boldsymbol v_1,\dots,\boldsymbol v_{|E|-|V|}]\cdot \boldsymbol x \in A(\mathbb R^V);$
and by Proposition \ref{tetrasym} and Proposition \ref{prop: horosphere}, we have 
\begin{equation}\label{F=}
    F(P\cdot \boldsymbol x)=F(O\cdot \boldsymbol x).
\end{equation}
Then by (\ref{b-TV2}), (\ref{AO}), (\ref{OP}) and (\ref{F=}), we have
\begin{equation}\label{sP}
 \begin{split}
   \mathrm{TV_b}(M,\mathcal T)= & \int _ X  F(\boldsymbol a)\big|\det[A,O]\big|dm_X(\boldsymbol a)\\
   = & \int _ {\mathbb R^{|E|-|V|}} F(O\cdot \boldsymbol x)\big|\det[A,O]\big|d\boldsymbol x\\
    = & \int _ {\mathbb R^{|E|-|V|}} F(P\cdot \boldsymbol x)\big|\det[A,P]\big|d\boldsymbol x
 = \int_U F( \boldsymbol a)\frac{\big|\det[A,P]\big|}{\sqrt{\det(P^TP)}}dm_U(\boldsymbol a).
 \end{split}   
\end{equation}

Finally, for the case that $P$ consists of a general basis of $U,$ we let $P'$ be an $|E|\times (|E|-|V|)$-matrix in the previous case. Then each column of $P$ is a linear combination of the column $P'.$ As a consequence, there is an invertible $(|E|-|V|)\times (|E|-|V|)$-matrix $G$ such that $P = P'\cdot G.$ Hence 
$$\det[A,P]=\det G\det[A,P']\quad\text{and}\quad \sqrt{\det(P^TP)}=\big|\det G\big|\sqrt{\det(P'^T P')},$$
and as a consequence
\begin{equation}\label{det=}
\frac{\big|\det[A,P]\big|}{\sqrt{\det(P^TP)}}=\frac{\big|\det[A,P']\big|}{\sqrt{\det(P'^TP')}}.
\end{equation}
Putting (\ref{sP}) and (\ref{det=}) together, we have 
$$   \mathrm{TV_b}(M,\mathcal T)= \int_U F( \boldsymbol a)\frac{\big|\det[A,P']\big|}{\sqrt{\det(P'^TP')}}dm_U(\boldsymbol a)= \int_U F( \boldsymbol a)\frac{\big|\det[A,P]\big|}{\sqrt{\det(P^TP)}}dm_U(\boldsymbol a).$$
\end{proof}

\begin{proof}[Proof of Theorem \ref{topinv}]
For $k\in \{1,2\},$ let $E_k$ and $T_k$ respectively be the sets of edges and tetrahedra in $\mathcal T_k,$ and let $e^*$ be the new edge produced by the $2$-$3$ Pachner Move so that $E_2=\{e^*\}\cup E_1.$ Let $\Delta_1,\Delta_2\in T_1$ and $\Delta_1^*,\Delta_2^*,\Delta_3^*\in T_2$ so that $T_2=\{\Delta_1^*,\Delta_2^*,\Delta_3^*\}\cup T_1 \setminus \{\Delta_1,\Delta_2\}.$ 

Let $\iota_1: \big(\frac{Q}{2}+\mathbf i\mathbb R\big)^{E_1}\to \big(\frac{Q}{2}+\mathbf i\mathbb R\big)^{E_2}$ be the inclusion  defined by $\iota_1(\boldsymbol a)=(\boldsymbol a, 0)$
 and let $\iota_2: \frac{Q}{2}+\mathbf i\mathbb R \to \big(\frac{Q}{2}+\mathbf i\mathbb R\big)^{E_2}$ be the inclusion  defined by $\iota_2(a)=(\boldsymbol 0,a).$ 
For $k\in\{1,2\},$ let $A_k$ be the $|E_k|\times |V|$ adjacency matrix of $\mathcal T_k.$ Then 
\begin{equation}\label{A2}
    A_2= \left[\begin{matrix}
 A_1 \\
 \boldsymbol r
\end{matrix}\right]
\end{equation} for some row  $\boldsymbol r$ with $|V|$ components coming from the edge $e^*.$ Let $X$ be the orthogonal complement of $A_1(\mathbb R^V)$ in $\big(\frac{Q}{2}+\mathbf i\mathbb R\big)^{E_1},$ and let $O$ be the $|E_1|\times(|E_1|-|V|)$-matrix consisting of an orthonormal basis $\{\boldsymbol o_i\}_{i\in\{1,\dots,|E_1|-|V|\}}.$
Let $U=\iota_1(X)+\iota_2\big(\frac{Q}{2}+\mathbf i\mathbb R\big).$ Then $U+A_2(\mathbb R^V)=\big(\frac{Q}{2}+\mathbf i\mathbb R\big)^{E_2}$ and $U\cap A_2(\mathbb R^V)=\{\boldsymbol 0\};$ and $\{\iota_1(\boldsymbol o_i)\}_{i\in \{1,\dots,|E_1|-|V|\}}\cup \{\iota_2\big(\frac{Q}{2}+\mathbf i\big)\}$ form a basis of $U.$ Let $P$ be the $|E_2|\times(|E_2|-|V|)$-matrix consisting of these basis vectors as the columns. Then $U$ and $P$ satisfy the conditions of Lemma \ref{P}, and
\begin{equation}\label{PO1}
P=\left[\begin{matrix}
    O & \boldsymbol 0\\
    \boldsymbol 0& 1
\end{matrix}\right],    
\end{equation}
where the the top-right $\boldsymbol 0$ is a column of $|E_1|$ components and the bottom-left $\boldsymbol 0$ is a row of $|E_1|-|V|$ components. As a consequence, we have  
\begin{equation}\label{detPTP}
\det(P^TP)=\det (O^TO)=1.
\end{equation}
Putting (\ref{A2}) and (\ref{PO1}) together, we have 
$$[A_2,P] = \left[\begin{matrix}
    A_1 & O & \boldsymbol 0\\
    \boldsymbol r & \boldsymbol 0 & 1
\end{matrix}  \right],$$
and hence 
\begin{equation}\label{detA2P}
    \big|\det [A_2,P]\big|= \Bigg|\det \left[\begin{matrix}
    A_1 & O & \boldsymbol 0\\
    \boldsymbol 0 & \boldsymbol 0 & 1
\end{matrix}  \right]\Bigg|=\big|\det [A_1, O]\big|=\sqrt{\det(A_1^TA_1)}.
\end{equation}

Now, by (\ref{b-TV2}), we have 
 \begin{equation*}
  \begin{split}
    \mathrm{TV_b}(M,\mathcal T_1)=& \Big(\frac{1}{2\pi b}\Big)^{|V|} \int _ X \prod_{e\in E_1} |e|_{\boldsymbol{a}}\prod_{\Delta\in T_1}|\Delta|_{\boldsymbol a}\sqrt{\det(A_1^T A_1)}dm_X(\boldsymbol a)\\
=& \Big(\frac{1}{2\pi b}\Big)^{|V|}  \int _ X \bigg(\int_{\frac{Q}{2}+\mathbf i\mathbb R}|T_b(2a_{e^*})|^2|\Delta_1^*|_{(\boldsymbol a,a_{e^*})}|\Delta_2^*|_{(\boldsymbol a,a_{e^*})}|\Delta_3^*|_{(\boldsymbol a,a_{e^*})}d\mathrm{Im}a_{e^*}\bigg)\\
&\quad\quad\quad\quad\quad\quad\quad\quad\quad\quad\quad\quad \prod_{e\in E_1} |e|_{\boldsymbol{a}}\prod_{\Delta\in T_1\setminus \{\Delta_1,\Delta_2\}}|\Delta|_{\boldsymbol a}\sqrt{\det(A_1^T A_1)}dm_X(\boldsymbol a)\\
=& \Big(\frac{1}{2\pi b}\Big)^{|V|}  \int _ {X\oplus \big(\frac{Q}{2}+\mathbf i\mathbb R\big)} \prod_{e\in E_2} |e|_{(\boldsymbol{a},a_{e^*})}\prod_{\Delta\in T_2}|\Delta|_{(\boldsymbol a,a_{e^*})}\sqrt{\det(A_1^T A_1)}dm_X(\boldsymbol a)d\mathrm{Im}a_{e^*}\\
=& \Big(\frac{1}{2\pi b}\Big)^{|V|}  \int _ U \prod_{e\in E_2} |e|_{\boldsymbol{a}}\prod_{\Delta\in T_2}|\Delta|_{\boldsymbol a}\sqrt{\det(A_1^T A_1)}dm_U(\boldsymbol a)\\
=& \Big(\frac{1}{2\pi b}\Big)^{|V|}  \int _ U \prod_{e\in E_2} |e|_{\boldsymbol{a}}\prod_{\Delta\in T_2}|\Delta|_{\boldsymbol a}\frac{\big|\det[A_2,P]\big|}{\sqrt{\det(P^T P)}}dm_U(\boldsymbol a) = \mathrm{TV}_b(M,\mathcal T_2),
\end{split}   
 \end{equation*}
where the second equality comes from Proposition \ref{32}, the third equlity comes from Fubini's Theorem, the penultimate equality comes from (\ref{detPTP}) and (\ref{detA2P}), and the last equality comes from Lemma \ref{P}.
\end{proof}


\section{Asymptotics of $\mathrm{TV}_b(M,\mathcal T)$}

Let $M$ be a hyperbolic $3$-manifold with cusps and  with an ideal triangulation $\mathcal T.$ Let $E$ and $T$ respectively be the sets of edges and tetrahedra of $\mathcal T.$ For $e\in E,$ let $x_e=-2\mathbf i\pi b\big( a_e-\frac{Q}{2}\big)$ so that $a_e=\frac{Q}{2}+\mathbf i \frac{x_e}{2\pi b}.$ For $\Delta\in T$ with edges $e_{\Delta,1},\dots,e_{\Delta,6},$ let $\boldsymbol x_\Delta=(x_{e_{\Delta,1}},\dots,x_{e_{\Delta,6}})$ and let $\zeta_\Delta=\pi b (2Q-u_\Delta)+\frac{\mathbf i}{2}\sum_{k=1}^6x_{e_{\Delta,k}}.$ Let $V_b\big(\boldsymbol x_\Delta,\zeta_\Delta\big)=V_{\boldsymbol x_\Delta,b}(\zeta_\Delta)$
be as defined in (\ref{Vxb}) and let $V\big(\boldsymbol x_\Delta,\zeta_\Delta\big)=V_{\boldsymbol x_\Delta}(\zeta_\Delta)$
be as defined in (\ref{Vx}). Let $V$ be the set of boundary components of $M,$ and recall again that the $|E|\times |V|$-matrix $A$ is the adjacency matrix of the edges and the boundary components given by $A_{e,v}=|e\cap v|,$ the number of points of intersections of the edge $e$ with the boundary component $v.$  The $\mathbb R^V$-action on $\mathbb R^E$ induced by the change of horoshperes is given by $\boldsymbol u\cdot \boldsymbol x = \boldsymbol x + A \boldsymbol u.$ We also consider $A$ as a linear map from $\mathbb R^V$ to $\mathbb R^E,$ and by abuse of notation let $X\subset \mathbb R^E$ be the orthogonal complement of $A(\mathbb R^V).$

Then by (\ref{b-TV2}),  we have
\begin{equation}\label{TVint}
\begin{split}
\mathrm{TV}_b(M,\mathcal T)= \frac{\sqrt{\det(A^TA)}}{2^{|E|}(\pi b )^{|E|+|T|}}
\int_{\Gamma} e^{\sum_{e\in E}x_e}\cdot\exp\Bigg(\frac{\mathcal V_b(\boldsymbol x, \boldsymbol \zeta)}{2\pi \mathbf i b^2}\Bigg)dm_X(\boldsymbol x)  d \boldsymbol \zeta,
\end{split}
\end{equation}
where the contour 
$$\Gamma=  \bigg\{ (\boldsymbol x,\boldsymbol \zeta)\in X \times \mathbb C^T   \ \bigg|\ \zeta_\Delta\in\Gamma_{\boldsymbol x_\Delta} \text{ for all }\Delta\text{ in }T\bigg\}$$
and $\Gamma_{\boldsymbol x_\Delta}$ is the contour of integral in (\ref{idealb6j}) that 
 will be specified later, $dm_X$ is the Lebesgue measure on $X$ obtained by restricting the Lebesgue measure $d\boldsymbol x=\prod _{e\in E}dx_e$ on $\mathbb R^E,$ $d\boldsymbol \zeta=\prod _{\Delta\in T}d\zeta_\Delta,$ and
\begin{equation}\label{Wnu}
\begin{split}
\mathcal V_b(\boldsymbol x, \boldsymbol \zeta)=&2\pi\mathbf i \sum_{e\in E}x_e+\sum_{\Delta\in T}V_b\big(\boldsymbol x_\Delta,\zeta_\Delta\big)\\
=&\mathcal V(\boldsymbol x, \boldsymbol \zeta)+\kappa(\boldsymbol x,\boldsymbol \zeta)b^2+\nu_b(\boldsymbol x,\boldsymbol \zeta)b^4
\end{split}
\end{equation}
with
\begin{equation}\label{calV}
    \mathcal V(\boldsymbol x, \boldsymbol \zeta)=2\pi \mathbf i \sum_{e\in E}x_e+\sum_{\Delta\in T}V\big(\boldsymbol x_\Delta,\zeta_\Delta\big),
\end{equation}

 $$\kappa(\boldsymbol x,\boldsymbol\zeta)=\sum_{\Delta\in T}\kappa_{\boldsymbol x_\Delta}(\zeta_\Delta)$$
 with $\kappa_{\boldsymbol x}$ as in (\ref{kappa}),  and $$\nu_b(\boldsymbol x,\boldsymbol\zeta)=\sum_{\Delta\in T}\nu_{\boldsymbol x_\Delta,b}(\zeta_\Delta)$$  with  $\nu_{\boldsymbol x,b}$  as in (\ref{nu}).

\subsection{Properties of relevant functions}

Suppose that $\mathcal T$ is geometric, and  $\boldsymbol x^*=(x_e^*)\in \mathbb R^E$ is a decorated edge-lengths vector of $(M,\mathcal T).$ By a change of horospheres if necessary, we may assume that $\boldsymbol x^* \in X.$ 

For $\Delta\in T$ with edges $e_{\Delta,1},\dots,e_{\Delta,6},$ let 
 $\boldsymbol x^*_{\Delta}=(x^*_{e_{\Delta,1}},\dots,x^*_{e_{\Delta,6}}).$ 
Guaranteed by Proposition \ref{critical2}, 
we let $\zeta^*_{\Delta}$ be the unique critical point of $V_{\boldsymbol x^*_{\Delta}}$ in 
$$D=\Big\{ \zeta\in\mathbb C\ \Big|\ 0 < \mathrm{Re}\zeta < \frac{\pi}{6}\Big\}.$$ 
Then we have the following

\begin{proposition}\label{cp} 
The function 
$$\mathcal V(\boldsymbol x, \boldsymbol \zeta)=2\pi\mathbf i\sum_{e\in E} x_e+\sum_{\Delta\in T}V\big(\boldsymbol x_\Delta,\zeta_\Delta\big)$$  has a critical point 
$$\boldsymbol z^*=(\boldsymbol x^* ,\boldsymbol \zeta^* )=\Big(\big(x^*_e\big)_{e\in E},\big(\zeta^*_{\Delta}\big)_{\Delta\in T}\Big)$$
in $X\times D^T$
with the critical value
$$\mathcal V(\boldsymbol z^*)=-2\mathbf i\mathrm{Vol}(M).$$ 
\end{proposition}

\begin{proof}  For $\Delta\in T$ with edges $e_{\Delta,1},\dots,e_{\Delta,6},$ let $x^*_{ \Delta}=(x^*_{ e_{\Delta,1}},\dots,x^*_{ e_{\Delta,6}}).$ 

We first have
\begin{equation}\label{xi=0b}
\frac{\partial  \mathcal V}{\partial \zeta_\Delta}\Big|_{(\boldsymbol x^* ,\boldsymbol \zeta^* )}=\frac{\partial V_{\boldsymbol x^*_{ \Delta}}(\zeta_\Delta)}{\partial \zeta_\Delta}\Big|_{\zeta^*_{ \Delta}}=0.
\end{equation}
Next,  let $W(\boldsymbol x_\Delta)=V(\boldsymbol x_\Delta,\zeta^*_\Delta)=V_{\boldsymbol x_\Delta}(\zeta^*_\Delta)$ be the function defined in (\ref{W}). Then for any edge $e$ of $\Delta,$ 
\begin{equation*}
\begin{split}
\frac{\partial W(\boldsymbol x_\Delta)}{\partial x_e}\Big|_{\boldsymbol x^*_{ \Delta}}=&\frac{\partial V(\boldsymbol x_\Delta,\zeta_\Delta)}{\partial x_e}\Big|_{(\boldsymbol x^*_{ \Delta},\zeta^*_{ \Delta})}+\frac{\partial V(\boldsymbol x_\Delta,\zeta_\Delta)}{\partial \zeta_\Delta}\Big|_{(\boldsymbol x^*_{ \Delta},\zeta^*_{ \Delta})}\cdot\frac{\partial \zeta^*_\Delta}{\partial x_e}\Big|_{\boldsymbol x^*_{ \Delta}}\\
=&\frac{\partial V(\boldsymbol x_\Delta,\zeta_\Delta)}{\partial x_e}\Big|_{(\boldsymbol x^*_{ \Delta},\zeta^*_{ \Delta})}.
\end{split}
\end{equation*}
As a consequence, by \eqref{co-sch2}, we have
\begin{equation*}
\frac{\partial V(\boldsymbol x_\Delta,\zeta_\Delta)}{\partial x_e}\Big|_{(\boldsymbol x^*_{ \Delta},\zeta^*_{ \Delta})}=
\frac{\partial W(\boldsymbol x_\Delta)}{\partial x_e}\Big|_{\boldsymbol x^*_{ \Delta}}=-\mathbf i \theta_{(\Delta,e)},
\end{equation*}
where
$\theta_{(\Delta,e)}$ is the  dihedral angle of the $\boldsymbol x^*_{ \Delta}$ at the edge $e.$ Then for each $e\in E,$ 
\begin{equation}\label{alpha=0b}
\frac{\partial  \mathcal V}{\partial x_e}\Big|_{(\boldsymbol x^* ,\boldsymbol \zeta^* )}=2\pi\mathbf i+ \sum_{\Delta\in T}\frac{\partial V(\boldsymbol x_\Delta,\zeta_\Delta)}{\partial x_e}\Big|_{(\boldsymbol x^*_{ \Delta},\zeta^*_{ \Delta})}=\mathbf i\Big(2\pi-\sum_{\Delta\sim e} \theta_{(\Delta,e)}\Big)=0,
\end{equation}
where the last summation is over all the corners around $e,$ and the last equality comes from the fact that the sum of the dihedral angles around an edge equals $2\pi.$ By (\ref{xi=0b}) and (\ref{alpha=0b}), $(\boldsymbol x^* ,\boldsymbol \zeta^* )$ is a critical point of $ \mathcal V$ in $\mathbb R^E\times D^T,$ hence is a critical point of $\mathcal V$ in $X\times D^T.$

 Finally, for the critical value, by Proposition \ref{critical2}, we have
\begin{equation*}
\begin{split}
\mathcal V(\boldsymbol x^*,\boldsymbol \zeta^*)=&2\pi \mathbf i \sum_{e\in E} x^*_e-2\mathbf i\sum_{\Delta\in T}\Big(\mathrm{Vol}(\Delta)+\sum_{e\sim \Delta}\frac{\theta_{(\Delta, e)}x^*_e}{2}\Big)\\
=& -2\mathbf i\sum_{\Delta\in T}\mathrm{Vol}(\Delta)+\mathbf i\sum_{e\in E} \Big(2\pi -\sum_{\Delta\sim e} \theta_{(\Delta,e)}\Big)x^*_e\\
=& -2\mathbf i\mathrm{Vol}(M),
\end{split}
\end{equation*}
where the last summation in the first row is over all the edges $e$ in $\Delta,$ and the last summation in the second row is over all the corners around $e.$ 
\end{proof}

In the rest of this section, we will let 
$$\Gamma^* = X\times \prod_{\Delta\in T} \Gamma^*_\Delta,$$ where
$\Gamma^*_\Delta = \Big\{ \zeta\in D \ \Big|\ \mathrm{Re}\zeta=\mathrm{Re}\zeta^*_\Delta \Big\}$
 for each $\Delta\in T.$ Then 
  $$\boldsymbol z^*=(\boldsymbol x^*,\boldsymbol \zeta^*)\in \Gamma^*.$$

\begin{proposition}\label{cd}
The function $\mathrm{Im}\mathcal V$ is strictly concave down on $\Gamma^*,$ hence achieves the absolute maximum at the critical point $\boldsymbol z^*.$
  \end{proposition}

\begin{proof}
For $(\boldsymbol x,\boldsymbol \zeta)\in \Gamma^*=X\times \prod_{\Delta\in T}\Gamma_\Delta^*,$ we let 
$$\mathcal L(\boldsymbol x,\boldsymbol \zeta) = 2\pi \mathbf i \sum_{e\in E}x_e+\sum_{\Delta\in T}\Big(V(\boldsymbol x_\Delta, \zeta_\Delta)-\sum_{i=1}^3L(\zeta_\Delta-\mathbf i y_{\Delta, i})\Big)$$
and let 
$$\mathcal N(\boldsymbol x,\boldsymbol \zeta) = \sum_{\Delta\in T}\sum_{i=1}^3 L(\zeta_\Delta-\mathbf i y_{\Delta, i}),$$
where for each $\Delta\in T$ and $\boldsymbol x_\Delta = (x_{\Delta,1},\dots,x_{\Delta,6})\in \mathbb R^6,$ $y_{\Delta,i}=\frac{x_{\Delta,i}+x_{\Delta,i+3}}{2}$ for each $i\in\{1,2,3\}.$ Then by (\ref{calV}), we have 
$$\mathcal V(\boldsymbol x,\boldsymbol \zeta) = \mathcal L(\boldsymbol x,\boldsymbol \zeta) + \mathcal N(\boldsymbol x,\boldsymbol \zeta). $$

By (\ref{Vx}),  $\mathrm{Im}\mathcal L(\boldsymbol x,\boldsymbol \zeta)$ is a linear function on $\Gamma^*.$ 
\medskip

Therefore,  it suffices to show that $\mathrm{Im}\mathcal N(\boldsymbol x,\boldsymbol \zeta)$ is strictly concave down  on $\Gamma^*.$ 
To this end, for each $\Delta\in T,$ we let 
$$\xi_\Delta = \zeta_\Delta - \frac{\mathbf i}{6}\sum_{k=1}^6x_{\Delta,k};$$ and let 
$\boldsymbol s_\Delta= (s_{\Delta,1},s_{\Delta,2},s_{\Delta,3}),$ 
where for each $i\in\{1,2,3\},$ 
$$s_{\Delta,i}=y_{\Delta,i}-\frac{1}{6}\sum_{k=1}^6x_{\Delta,k}.$$
Then we have 
$s_{\Delta,1}+s_{\Delta,2}+s_{\Delta,3}=0.$ 
Let $F=\big\{(s_1,s_2,s_3)\in\mathbb R^3\ \big|\ s_1+s_2+s_3=0\big\}.$ We define the linear map 
$$S: \Gamma^* \to \prod_{\Delta \in T} \big(F\times \Gamma_\Delta^*\big)$$
by 
$$S(\boldsymbol x,\boldsymbol \zeta)\doteq (\boldsymbol s_\Delta, \xi_\Delta)_{\Delta\in T};$$
and define the function $\mathcal U$ on $\prod_{\Delta \in T} \big(F\times \Gamma_\Delta^*\big)$ by
$$\mathcal U\Big((\boldsymbol s_\Delta,\xi_\Delta)_{\Delta\in T}\Big)\doteq \sum_{\Delta\in T}\sum_{i=1}^3L(\xi_\Delta-\mathbf i s_{\Delta,i}).$$
Then, as $\zeta_\Delta-\mathbf i y_{\Delta,i}=\xi_\Delta-\mathbf i s_{\Delta,i}$ for each $\Delta\in T$ and $i\in\{1,2,3\},$
we have 
$$\mathcal N(\boldsymbol x,\boldsymbol \zeta)= \mathcal U \circ S (\boldsymbol x,\boldsymbol \zeta)$$
for each $(\boldsymbol x,\boldsymbol \zeta)\in \Gamma^*.$ As a consequence, to show that $\mathrm{Im}\mathcal N$ is strictly concave down on $\Gamma^*,$ it suffices to show that:
\begin{enumerate}[(1)]
    \item the linear map 
$S: \Gamma^* \to \prod_{\Delta \in T} \big(F\times \Gamma_\Delta^*\big)$ is injective, and

    \item  $\mathrm{Im}\mathcal U$ is strictly concave down on $\prod_{\Delta\in T} \big(F\times \Gamma_\Delta^*\big).$
\end{enumerate}

For (1), let $B$ be the linear map in Lemma \ref{inj}. Then under the identification of $\big(\frac{Q}{2}+\mathbf i \mathbb R\big)^E$ with $\mathbb R^E$ that identifies  $\boldsymbol a$ with $\mathrm{Im}\boldsymbol a,$ the identification of $\prod_{\Delta\in T}\Gamma_\Delta^*$ with $\mathbb R^T$ that identifies $\boldsymbol \zeta$ with $\mathrm{Im}\boldsymbol \zeta,$
and the identification of $\prod_{\Delta\in T}\big(F\times \Gamma_\Delta^*\big)$ with $F^T\times \mathbb R^T$  that identifies $\big((\boldsymbol s_\Delta,\xi_\Delta)_{\Delta\in T}\big)$ with $\big((\boldsymbol s_\Delta)_{\Delta\in T},(\mathrm{Im}\xi_\Delta)_{\Delta\in T}\big),$ the linear map 
 $S: X\times \mathbb R^E \to F^T\times \mathbb R^E$ is blocked  with the top-left block the restriction of $\frac{1}{2}B$ on $X,$ the top-right block the zero matrix and the bottom-right block the $|T|\times |T|$ identity matrix. Then by Lemma \ref{inj}, $S$ is injective. 
\medskip

For (2), it suffices to show that for each $\Delta\in T,$ the function $\sum_{i=1}^3\mathrm{Im}L(\xi_\Delta-\mathbf i s_{\Delta,i})$ is strictly concave down on $F \times \Gamma_\Delta^*,$ which is equivalent to showing that the function 
$$U\big(s_{\Delta,1},s_{\Delta,2},\mathrm{Im}\zeta_\Delta\big)\doteq \mathrm{Im}L\big(\xi_\Delta-\mathbf i s_{\Delta,1}\big)+\mathrm{Im}L\big(\xi_\Delta-\mathbf i s_{\Delta,2}\big)+\mathrm{Im}L\big(\xi_\Delta+\mathbf i (s_{\Delta,1}+s_{\Delta,2})\big)$$ is strictly concave down on $\mathbb R^3.$ By a direct computation, we have 
that the Hessian matrix  
$$\mathrm{Hess}U= K^T \cdot D \cdot K,$$
where 
$$K=\left[ \begin{matrix} 
    1 & 0 & -1 \\
    0 & 1 & -1 \\
    1 & 1 & 1 \\
\end{matrix}
\right] $$
is an invertible matrix, and $D$ is the following diagonal matrix
$$\left[ \begin{matrix} 
    \frac{-2\sin (2\mathrm{Re}\zeta_\Delta)}{\cosh (2\mathrm{Im}\zeta_\Delta-2s_{\Delta,1})-\cos (2\mathrm{Re}\zeta_\Delta)} & 0 & 0 \\
    0 &  \frac{-2\sin (2\mathrm{Re}\zeta_\Delta)}{\cosh (2\mathrm{Im}\zeta_\Delta-2s_{\Delta,2})-\cos (2\mathrm{Re}\zeta_\Delta)} & 0 \\
    0 & 0 &  \frac{-2\sin (2\mathrm{Re}\zeta_\Delta)}{\cosh (2\mathrm{Im}\zeta_\Delta+2s_{\Delta,1}+2s_{\Delta,2})-\cos (2\mathrm{Re}\zeta_\Delta)} \\
\end{matrix}
\right].$$
Then by Proposition \ref{critical2} that $\mathrm{Re}\zeta_\Delta = \mathrm{Re}\zeta_\Delta^*\in \big(0, \frac{\pi}{12}\big],$  all the diagonal entries above are negative. As a consequence, $\mathrm{Hess}U$ is negative definite, and $U$ is strictly concave down on $\mathbb R^3.$
\end{proof}

\begin{proposition}\label{nond}
The critical point  $\boldsymbol z^*$ of $\mathcal V$ is non-degenerate.
\end{proposition}

\begin{proof}
This is  a consequence of Proposition \ref{cd} and \cite{L} that if the imaginary part of a complex symmetric matrix is negative definite, then the matrix is invertible. 
\end{proof}

\begin{proposition}\label{bound2} 

\begin{enumerate}[(1)]
\item  There exists a constant $K>0$ such that, for all unit vector 
$\mathbf u$ in  the tangent space of $\Gamma^*$ at each of its points, 
 the directional derivative 
$$D_{\mathbf u}\mathrm{Im}\kappa(\boldsymbol x,\boldsymbol \zeta)<K$$
for all for $(\boldsymbol x,\boldsymbol \zeta)$ in $\Gamma^*.$

\item Let $\mathcal D$ be the $\delta$-neighborhood of $\Gamma^*$ in $\mathbb C^E\times \mathbb C^T$ for some $\delta>0.$ Then for $b>0$ sufficiently small,  there exists a constant $N>0$ independent of $b$ such that 
$$\mathrm{Im}\mathcal \nu_b(\boldsymbol x, \boldsymbol \zeta)\leqslant \big|\nu_b(\boldsymbol x, \boldsymbol \zeta)\big|<N$$
for all $(\boldsymbol x,\boldsymbol \zeta)\in  \mathcal  D.$  
\end{enumerate}
\end{proposition}

The proof of Proposition \ref{bound2} follows verbatim that of \cite[Proposition 5.7 and Proposition 5.8]{LMSWY}.

\subsection{Proof of Theorem \ref{vol}}

\begin{proof}[Proof of Theorem \ref{vol}] 
In (\ref{TVint}), we let $\Gamma = \Gamma^*,$ and let $\mathcal D$ be a $\delta$-neighborhood of $\Gamma^*$ in $\mathbb C^E\times \mathbb C^T$ for some $\delta>0.$ 

Let  $d>0$ be sufficiently small so that the $d$-neighborhood $B_{d}$ of $\boldsymbol z^*$ 
 lies entirely in $\mathcal D.$

By Proposition \ref{cp} and Proposition \ref{cd},  there is an $\epsilon>0$ sufficiently small so that 
 \begin{equation}\label{ImV}
\mathrm{Im} \mathcal V(\boldsymbol z^* + l\mathbf u )< -2\mathrm{Vol}(M)-4\epsilon
\end{equation}
and for $l>d$ and for any unit tangent vector $\mathbf u$ of $\Gamma^*;$ and there is an $L>d$ such that 
\begin{equation}\label{Dv11}
 D_{\mathbf u} \mathrm{Im} \mathcal V(\boldsymbol z^*+ l\mathbf u)<-2\epsilon
\end{equation}
for $l>L$ and for any unit tangent vector $\mathbf u$ of $\Gamma^*.$

Let
$$\Gamma^*_d=\Gamma^*\cap B_d=\Big\{ (\boldsymbol x,\boldsymbol \zeta) \in \Gamma^*\ \Big|\ \big| (\boldsymbol x,\boldsymbol \zeta)-\boldsymbol z^* \big|<d\Big\}$$
and let 
$$\Gamma^*_L=\Big\{ (\boldsymbol x,\boldsymbol \zeta) \in \Gamma^*\ \Big|\ \big| (\boldsymbol x,\boldsymbol \zeta)-\boldsymbol z^* \big|\leqslant L\Big\}.$$
We will show that, as $b\to 0,$
\begin{enumerate}[(I)]
\item  
\begin{equation*}
\begin{split}
    \int_{\Gamma^*_d} e^{\sum_{e\in E}x_e}\cdot\exp\bigg(\frac{\mathcal V_b(\boldsymbol x, \boldsymbol \zeta)}{2\pi \mathbf i b^2}\bigg)&dm_X(\boldsymbol x)  d \boldsymbol \zeta\\
= & b^{|E|-|V|+|T|}\cdot C(\boldsymbol z^*)\cdot  e^{-\frac{\mathrm{Vol}(M)}{\pi b^2}}\big(1+O(b^2)\big)
\end{split}
\end{equation*}
for some non-zero $C(\boldsymbol z^*)$ which will be specified later.

\item 
$$
\bigg|\int_{\Gamma^*_L\setminus \Gamma^*_d} e^{\sum_{e\in E}x_e}\cdot\exp\bigg(\frac{\mathcal V_b(\boldsymbol x, \boldsymbol \zeta)}{2\pi \mathbf i b^2}\bigg)dm_X(\boldsymbol x)  d \boldsymbol \zeta\bigg|\leqslant  O\bigg( e^{-\frac{\mathrm{Vol}(M)+\epsilon_1}{\pi b^2}}\bigg)
$$
for some $\epsilon_1>0.$

\item

$$
\bigg|\int_{\Gamma^*\setminus \Gamma^*_L} e^{\sum_{e\in E}x_e}\cdot\exp\bigg(\frac{\mathcal V_b(\boldsymbol x, \boldsymbol \zeta)}{2\pi \mathbf i b^2}\bigg)dm_X(\boldsymbol x)  d \boldsymbol \zeta\bigg|\leqslant  O\bigg( e^{-\frac{\mathrm{Vol}(M)+\epsilon_2}{\pi b^2}}\bigg)
$$
for some $\epsilon_2>0.$ 
\end{enumerate}
\medskip

For (I), we claim that  all the conditions of Proposition \ref{saddle} are satisfied by letting $\hbar=b^2,$ $D=B_d,$ $f=\frac{\mathcal V}{2\pi \mathbf i},$ $g=\exp\big(\sum_{e\in E}x_e+\frac{\kappa}{2\pi \mathbf i }\big),$ $f_\hbar=\frac{\mathcal V+\nu_b b^4}{2\pi \mathbf i},$ $\upsilon_\hbar=\frac{\nu_{b}}{2\pi \mathbf i},$ $S=\Gamma^*_d$ and $c=\boldsymbol z^*.$ 

Indeed, by Proposition \ref{cp}, $\boldsymbol z^*$ is a critical point of $f=\frac{\mathcal V}{2\pi \mathbf i}$ in $B_{d},$ hence condition (i) is satisfied. 
By Propositions \ref{cp} and \ref{cd}, $\boldsymbol z^*$ is the unique maximum point of $\mathrm{Re}f=\frac{\mathrm{Im}\mathcal V}{2\pi}$ on $\Gamma^*_d,$ hence condition (ii) is satisfied. 
By Proposition \ref{nond}, condition (iii) is satisfied. 
Since $g$ is an exponential function, it is non-zero, and condition (iv) is satisfied.
For condition (v), by Proposition \ref{bound2} (2),  $|\upsilon_{\hbar}(\zeta)|=\big|\frac{\nu_{b}(\zeta)}{2\pi \mathbf i}\big|<\frac{N}{2\pi}$ on $B_{d}.$ 
For condition (vi), since $\Gamma^*$ is a straight line, it is a smooth embedding near $\boldsymbol z^*.$  

Since all the conditions are satisfied, by Proposition \ref{saddle}, Proposition \ref{cp} and Proposition \ref{nond}, we have as $b\to 0,$
\begin{equation*}
\begin{split}
    \int_{\Gamma^*_d} e^{\sum_{e\in E}x_e}\cdot\exp\bigg(\frac{\mathcal V_b(\boldsymbol x, \boldsymbol \zeta)}{2\pi \mathbf i b^2}\bigg)&dm_X(\boldsymbol x)  d \boldsymbol \zeta\\
    =& b^{|E|-|V|+|T|}\cdot C(\boldsymbol z^*)\cdot  e^{-\frac{\mathrm{Vol}(M)}{\pi b^2}}\big(1+O(b^2)\big),
\end{split}
\end{equation*}
where
$$C(\boldsymbol z^*) = \frac{(2\pi)^{\frac{|E|-|V|+|T|}{2}}e^{\sum_{e\in E}x^*_e}e^{\frac{\kappa(\boldsymbol z^*)}{2\pi \mathbf i} }}{\sqrt{\det\Big(-\mathrm{Hess}\frac{\mathcal V}{2\pi \mathbf i}(\boldsymbol z^*)\Big)}}\neq 0,$$
and $\mathrm{Hess}\frac{\mathcal V}{2\pi \mathbf i}$ is  the Hessian matrix of the restriction of $\frac{\mathcal V}{2 \pi\mathbf i}$ on $\Gamma^*.$
This completes the proof of (I). 
\medskip

For (II) and (III), we have
\begin{equation*}
\begin{split}
& \bigg|\int_{\Gamma^*_L\setminus \Gamma^*_d} e^{\sum_{e\in E}x_e}\cdot\exp\bigg(\frac{\mathcal V_b(\boldsymbol x, \boldsymbol \zeta)}{2\pi \mathbf i b^2}\bigg)dm_X(\boldsymbol x)  d \boldsymbol \zeta\bigg|\\
\leqslant & \int_{\Gamma^*\setminus \Gamma^*_d}\exp\bigg(\frac{\mathrm{Im}\mathcal V(\boldsymbol  x,\boldsymbol \zeta)+\big(\mathrm{Im}\kappa(\boldsymbol x,\boldsymbol \zeta)+ 2\pi \sum_{e\in E}x_e\big)b^2+\mathrm{Im}\nu_b(\boldsymbol  x,\boldsymbol \zeta)b^4}{2\pi b^2} \bigg)|dm_X(\boldsymbol x)  d \boldsymbol \zeta|.
\end{split}
\end{equation*}
\smallskip

For  (II), by Proposition \ref{bound2} (2), there is a  $b_1>0$ such that 
\begin{equation}\label{last22}
\mathrm{Im}\nu_b(\boldsymbol x,\boldsymbol \zeta)b^4<Nb^4<\epsilon 
\end{equation}
for all $b<b_1$ and for all $(\boldsymbol x,\boldsymbol \zeta)\in\Gamma^*;$  and together with  (\ref{ImV}), we have
\begin{equation}\label{last33}
\mathrm{Im}\mathcal V(\boldsymbol  x, \boldsymbol \zeta)+\mathrm{Im}  \nu_{b}(\boldsymbol  x,\boldsymbol \zeta)b^4< -2\mathrm{Vol}(M) -3\epsilon
\end{equation}
for all $b<b_1$ and  $(\boldsymbol x,\boldsymbol \zeta)\in\Gamma^*\setminus \Gamma^*_d.$ By the compactness of ${\Gamma^*_L}\setminus \Gamma^*_d,$ there exists a $C>0$ such that 
\begin{equation}\label{Imkk}
\mathrm{Im}\kappa(\boldsymbol x,\boldsymbol \zeta)+2\pi \sum_{e\in E}x_e<C
\end{equation}
for all $(\boldsymbol x,\boldsymbol \zeta)\in\Gamma^*_L\setminus \Gamma^*_d.$ As a consequence of (\ref{last33}) and (\ref{Imkk}), we have
\begin{equation*}
\begin{split}
& \int_{\Gamma^*_L\setminus \Gamma^*_d}\exp\bigg(\frac{\mathrm{Im}\mathcal V(\boldsymbol  x,\boldsymbol \zeta)+\big(\mathrm{Im}\kappa(\boldsymbol x,\boldsymbol \zeta)+ 2\pi \sum_{e\in E}x_e\big)b^2+\mathrm{Im}\nu_b(\boldsymbol  x,\boldsymbol \zeta)b^4}{2\pi b^2} \bigg)|dm_X(\boldsymbol x)  d \boldsymbol \zeta|\\
<  & \Big(m(\Gamma^*_L)-m(\Gamma^*_d)\Big) e^{\frac{C}{2\pi}}\exp\bigg(\frac{-\mathrm{Vol}(M)-\epsilon }{\pi b^2}   \bigg )< O\Big(e^{\frac{-\mathrm{Vol}(M)-\epsilon_1}{\pi b^2}}\Big)
\end{split}
\end{equation*}
for any $\epsilon_1<\epsilon,$ where $m(\Gamma^*_L)$ and $m(\Gamma^*_d)$ are respectively the volume of $\Gamma^*_L$ and $\Gamma^*_d$ in the measure $dm_X(\boldsymbol x)  d \boldsymbol \zeta$ on $\Gamma^*.$ This completes the proof of (II). 
\medskip

For (III), there  is a $b_0\in (0, b_1)$ such that for all $b<b_0$ and for all unit tangent vector $\mathbf u=\big((u_e)_{e\in E},(u_\Delta)_{\Delta\in T}\big)$ of $\Gamma^*,$
\begin{equation}\label{ImKK}
\Big(\mathrm{Im}\kappa(\boldsymbol z^* + L\mathbf u) + 2\pi \sum_{e\in E} (x_e^*+ L u_e ) \Big)b^2<\epsilon,
\end{equation}
$(K+2\pi|E|)b^2<\epsilon$ and $Nb^4<\epsilon,$ where $K$ and $N$ are respectively the constants in Proposition \ref{bound2} (1) and (2). We claim that, for $(\boldsymbol x,\boldsymbol \zeta)\in\Gamma^*\setminus \Gamma^*_L$ and $b<b_0,$ 
\begin{equation}\label{cll}
\begin{split}
&\mathrm{Im}\mathcal V(\boldsymbol  x,\boldsymbol \zeta)+\Big(\mathrm{Im}\kappa(\boldsymbol x,\boldsymbol \zeta)+ 2\pi \sum_{e\in E}x_e\Big)b^2+\mathrm{Im}\nu_b(\boldsymbol  x,\boldsymbol \zeta)b^4\\
<&-2\big(\mathrm{Vol}(M)+\epsilon\big)-\epsilon\big(|(\boldsymbol x,\boldsymbol \zeta)-\boldsymbol z^*|-L\big),
\end{split}
\end{equation}
as a consequence of which, we have,
\begin{equation}\label{CII}
\begin{split}
& \int_{\Gamma^*\setminus \Gamma^*_L}\exp\bigg(\frac{\mathrm{Im}\mathcal V(\boldsymbol  x,\boldsymbol \zeta)+\big(\mathrm{Im}\kappa(\boldsymbol x,\boldsymbol \zeta)+ 2\pi \sum_{e\in E}x_e\big)b^2+\mathrm{Im}\nu_b(\boldsymbol  x,\boldsymbol \zeta)b^4}{2\pi b^2} \bigg)|dm_X(\boldsymbol x)  d \boldsymbol \zeta| \\
 < &  \exp\bigg(\frac{-\mathrm{Vol}(M)-\epsilon}{\pi b^2}\bigg)\int_{\Gamma^*\setminus \Gamma^*_L}\exp\bigg(\frac{-\epsilon\big(|(\boldsymbol x,\boldsymbol \zeta)-\boldsymbol z^*|-L\big)}{2\pi}\bigg)|dm_X(\boldsymbol x)  d \boldsymbol \zeta| \\
< & O\Big(e^{\frac{-\mathrm{Vol}(M)-\epsilon_2}{\pi b^2}}\Big)
\end{split}
\end{equation}
for any $\epsilon_2<\epsilon.$ 

For the proof of the claim, by (\ref{Dv11}) and the choice of $b_0,$ for $l>L,$ we have 
$$D_\mathbf u\bigg(\mathrm{Im}\mathcal V(\boldsymbol z^*+l\mathbf u)+\Big(\mathrm{Im}\kappa(\boldsymbol z^* + l\mathbf u) + 2\pi \sum_{e\in E} (x_e^*+ l u_e ) \Big)b^2\bigg)<-2\epsilon + (K+2\pi |E|)b^2<-\epsilon.$$
Together with the Mean Value Theorem, (\ref{ImV}) and (\ref{ImKK}), we have 
\begin{equation}\label{Bouu}
\begin{split}
&\mathrm{Im}\mathcal V(\boldsymbol  x,\boldsymbol \zeta)+\Big(\mathrm{Im}\kappa(\boldsymbol x,\boldsymbol \zeta)+ 2\pi \sum_{e\in E}x_e\Big)b^2\\
< & \mathrm{Im}\mathcal V(\boldsymbol z^*+ L\mathbf u) + \Big(\mathrm{Im}\kappa(\boldsymbol z^* + L\mathbf u) + 2\pi \sum_{e\in E} (x_e^*+ L u_e ) \Big)b^2 - \epsilon \big |  (\boldsymbol x,
\boldsymbol \zeta) - (\boldsymbol z^*+ L\mathbf u ) \big|\\
< & -2\mathrm{Vol }(M)-3\epsilon  -\epsilon\big(|(\boldsymbol x,
\boldsymbol \zeta)-\boldsymbol z^*|-L \big)
\end{split}
\end{equation}
for all $(\boldsymbol x,\boldsymbol \zeta )\in \Gamma^*\setminus \Gamma^*_L,$ where $\mathbf u$ is the unit tangent vector along the direction of $(\boldsymbol x,\boldsymbol \zeta)-\boldsymbol \zeta^*.$  Finally, putting (\ref{Bouu}) and (\ref{last22}) together, we have  (\ref{cll}) and the first  inequality in (\ref{CII}); and since 
$$|(\boldsymbol x,\boldsymbol \zeta)-\boldsymbol z^*|-L\to+\infty$$
as $(\boldsymbol x,\boldsymbol \zeta) \in \Gamma^*\setminus \Gamma^*_L$ approaches $\infty,$ we have the second inequality in (\ref{CII}). This completes the proof of (III).
\\

Putting (I), (II), (III) and (\ref{TVint}) together, we have as $b\to 0,$ 
$$\mathrm{TV}_b(M,\mathcal T) = \frac{1}{b^{|V|}}\cdot \mathbb T (\boldsymbol z^*)\cdot e^{-\frac{\mathrm{Vol}(M)}{\pi b^2}}\big(1+O(b^2)\big),
$$
where
$$\mathbb T(\boldsymbol z^*) =   \frac{\sqrt{\det (A^TA)}}{2^{\frac{|V|}{2}}\pi^{\frac{|E|+|V|+|T|}{2}}}\frac{e^{\sum_{e\in E}x^*_e} e^{ \frac{\kappa(\boldsymbol z^*)}{2\pi \mathbf i} }}{\sqrt{\det\Big(-\mathrm{Hess}\frac{\mathcal V}{2\pi \mathbf i}(\boldsymbol z^*)\Big)}}\neq 0,$$
from which the result follows. 
\end{proof}

\end{document}